\documentclass[a4paper,10pt]{article}
\usepackage[utf8]{inputenc}
\usepackage{amsmath}
\usepackage{amssymb}
\usepackage{amsthm}
\usepackage{amsfonts}
\usepackage{latexsym}
\theoremstyle{plain}
\newtheorem{thm}{Theorem}[section]
\newtheorem{cor}{Corollary}[section]
\newtheorem{lem}{Lemma}[section]
\newtheorem{prop}{Proposition}[section]
\theoremstyle{definition}
\newtheorem{defn}{Definition}[section]

\theoremstyle{remark}

\newtheorem{rem}{Remark}[section]

\title{Lower order asymptotics for Szeg\"{o} and Toeplitz kernels under Hamiltonian
circle actions}
\author{Roberto Paoletti\footnote{\noindent{\bf Address:}
Dipartimento di Matematica e Applicazioni, Universit\`a degli Studi
di Milano Bicocca, Via R. Cozzi 53, 20125 Milano,
Italy; {\bf e-mail}: roberto.paoletti@unimib.it }}
\date{}
\begin{document}

\maketitle

\begin{abstract}
We consider a natural variant of Berezin-Toeplitz quantization of compact
K\"{a}hler manifolds, in the presence of a Hamiltonian circle action lifting
to the quantizing line bundle. Assuming that the moment map is positive, we study
the diagonal asymptotics of the associated Szeg\"{o} and Toeplitz operators,
and specifically their relation to the moment map and to the geometry of a certain
symplectic quotient. When the underlying action is trivial and the moment map is taken to be
identically equal to one, this scheme coincides with the usual Berezin-Toeplitz quantization.
This continues previous work on near-diagonal scaling asymptotics of equivariant
Szeg\"{o} kernels in the presence of Hamiltonian torus actions.
\end{abstract}

\section{Introduction}

The object of this paper are the asymptotics of Szeg\"{o} and Toeplitz
operators in a non-standard version of the Berezin-Toeplitz quantization of
a complex projective K\"{a}hler manifold $(M,J,\omega)$. 

In Berezin-Toeplitz quantization,
one typically adopts as \lq quantum spaces\rq\,
the Hermitian spaces $H^0\left(M,A^{\otimes k}\right)$
of global holomorphic sections of higher powers of the polarizing line bundle
$(A,h)$; here $(A,h)$ is a positive, hence ample,
Hermitian holomorphic line bundle on $M$. Quantum observables, on the other
hand, correspond to Toeplitz operators associated to real $\mathcal{C}^\infty$
functions on $M$.

Here we assume given a Hamiltonian action 
$\mu^M$ of the circle group $U(1)=\mathbf{T}^1$ on $M$, with positive moment
map $\Phi$, and admitting a metric preserving linearization to 
$A$. It is then natural to replace the spaces $H^0\left(M,A^{\otimes k}\right)$
with certain new \lq quantum spaces\rq \,which arise by decomposing the Hardy space associated to
$A$ into isotypes for the action; these are generally not spaces
of sections of powers of $A$.  
One is thus led to consider analogues of the usual constructs of Berezin-Toeplitz
quantization. 
In particular, it is interesting to investigate
how the symplectic geometry of the underlying action, encapsulated in $\Phi$, 
influences the semiclassical asymptotics in this quantization scheme. 

This picture generalizes the usual Berezin-Toeplitz quantization 
of $(M,J,\omega)$, for one falls back on the standard case
by considering the trivial action of $\mathbf{T}^1$  
on $M$ with moment map $\Phi=1$. Then the lifted action is essentially fiberwise
scalar multiplication,
and the corresponding equivariant spaces are the usual spaces of global holomorphic
sections.

This theme was considered in \cite{pao-IJM}
for general Hamiltonian torus actions; the focus there was on near diagonal scaling
asymptotics of the associated equivariant Szeg\"{o} kernels. Here we shall restrict 
our analysis to 
circle actions, and investigate the lower order terms of these asymptotic expansion, 
as well as of their analogues for 
Toeplitz operators. 

In the usual standard setting of 
Berezin-Toeplitz quantization, a huge
amount of work has been devoted to these themes, 
involving a variety of approaches and techniques; see for example
(obviously with no pretense of being exhaustive) \cite{ae}, \cite{berezin}, 
\cite{bordermann-meinrenken-schlichenmaier}, \cite{bdm-g}, 
\cite{cahen-gutt-rawnsley-I}, \cite{cahen-gutt-rawnsley-II}, \cite{catlin},
\cite{charles}, \cite{englis0},
\cite{guillemin-star}, \cite{lu},
\cite{mm1}, \cite{mm2}, \cite{schlichenmaier}, 
\cite{sz}, \cite{tian},
\cite{zelditch-index-dynamics}, \cite{zelditch-theorem-of-Tian}, and references therein. 

In the present paper, we shall follow the general approach to quantization based 
on the (microlocal) analysis of the Szeg\"{o}
kernel on the circle bundle X of $A^\vee$; this train of thought was first introduced in 
the grounding work \cite{bdm-g},
and  afterwards explored
by many authors. 
We shall specifically also build on ideas and
results from
\cite{englis}, \cite{loi-1},
and \cite{karabegov-schlichenmaier}; in fact, the present
paper was considerably inspired by the concise approach in \cite{loi-1} to the derivation of the lower
order terms in the TYZ expansion for real-analytic metrics.

Now let us make our discussion more precise. Let $(M,J)$ be a connected
complex $d$-dimensional projective manifold, and let $A$ be an ample holomorphic
line bundle on $M$, with dual line bundle $A^\vee$ and projection 
$\widehat{\pi}:A^\vee\rightarrow M$.

There is an Hermitian metric $\ell_A$ on $A$ such that the unique
covariant derivative $\nabla_A$ on $A$ compatible with $\ell_A$ and the holomorphic structure
has curvature $\Theta_A=-2i\,\omega$, where $\omega$  is a K\"{a}hler form on $M$.
Then $dV_M=:\omega^{\wedge d}/d!$ is a volume form on $M$. 

Let $X\subseteq A^\vee$ be the unit circle bundle, 
with projection 
$\pi=\left.\widehat{\pi}\right|_X:X\rightarrow M$. 
Then $\nabla$ corresponds to a connection contact form $\alpha\in \Omega^1(X)$, 
such that $d\alpha=2\,\pi^*(\omega)$, and $dV_X=:(1/2\pi)\,\alpha\wedge \pi^*(dV_M)$
is a volume form on $X$. Let $L^2(X)=:L^2(X,dV_X)$, and identify functions
with densities and half-densities on $X$. Also, let $H(X)=:\ker \big(\overline{\partial}_b\big)
\cap L^2(X)$ be the Hardy space of $X$, where $\overline{\partial}_b$ is the 
Cauchy-Riemann operator on $X$.

Suppose that the action $\mu^M:\mathbf{T}^1\times M\rightarrow M$ is holomorphic with respect to
$J$ and Hamiltonian with respect to $2\,\omega$, with moment map $\Phi:M\rightarrow \mathbb{R}$; suppose
furthermore that $(\mu^M,\Phi)$ can be linearized to a
holomorphic action $\mu^A$ 
on $A$ leaving $\ell_A$ invariant. Then $\mathbf{T}^1$ acts on $X$ as a group
of contactomorphisms under the naturally induced action $\mu^X:\mathbf{T}^1\times X\rightarrow X$
lifting $\mu^M$.

Infinitesimally, the relation between $\mu^M$ and $\mu^X$ is as follows. 
Let $\partial/\partial \theta\in \mathfrak{X}(X)$ be the infinitesimal generator
of the action of $\mathbf{T}^1$ on $X$ given by
fiberwise scalar multiplication, $\mathrm{mult}:\big(e^{i\theta},x\big)\mapsto e^{i\theta}\cdot x$;
also, let $\xi_M\in \mathfrak{X}(M)$ 
be the infinitesimal generator of $\mu^M$, with horizontal lift $\xi_M^\sharp$ to $X$.
Then the infinitesimal generator $\xi_X\in  \mathfrak{X}(X)$ of $\mu^X$ is given by
\begin{equation}
 \label{eqn:infinitesimal-lift}
\xi_X=\xi_M^\sharp-\Phi\,\dfrac{\partial}{\partial \theta},
\end{equation}
where we write $\Phi$ for $\Phi\circ \pi$. Thus $\mu^X$ crucially depends on the choice of $\Phi$; for example, when 
$\mu^M$ is trivial choosing $\Phi=0$ yields the trivial action on $X$, while if $\Phi=1$ we
obtain the action 
\begin{equation}
\label{eqn:action-S-X}
  \nu^X:\mathbf{T}^1\times X\rightarrow X,\,\,\,\big(e^{i\theta},x\big)\mapsto e^{-i\theta}\cdot x.
\end{equation}

Since $\mu^X$ preserves $\alpha$ and is a lifting of the holomorphic action $\mu^M$, it leaves 
$H(X)$ invariant; therefore it determines a unitary action of $\mathbf{T}^1$ on $H(X)$. Thus
$H(X)$ equivariantly and unitarily decomposes into the Hilbert direct sum of its isotypes,
\begin{equation}
 \label{eqn:isotypes}
H^\mu_k(X)=:\left\{f\in H(X)\,:\,f\left(\mu^X_{g^{-1}}(x)\right)
=g^k\,f(x)\,\,\,\forall\, (g,x)\in \mathbf{T}^1\times X\right\},
\end{equation}
where $k\in \mathbb{Z}$.
If $\mu^M$ is trivial and $\Phi=1$, (\ref{eqn:isotypes}) is the standard
$k$-th equivariant Szeg\"{o} kernel $H_k(X)$, which is unitarily isomorphic
to $H^0\left(M,A^{\otimes k}\right)$ in a natural manner. However, in general $H^\mu_k(X)$
is not a space of sections of powers of $A$, and may even be 
infinite-dimensional. For example, if $\mu^M$ is trivial and $\Phi=0$ then $H^\mu_0(X)=H(X)$,
while $H^\mu_k(X)$ is the null space for $k\neq 0$.

Nonetheless, if $\Phi>0$ then $H^\mu_k(X)$ is finite-dimensional for any $k\in \mathbb{Z}$,
and is the null space if $k<0$ \cite{pao-IJM}; in particular, the orthogonal projector
$\Pi^\mu_k:L^2(X)\rightarrow H^\mu_k(X)$ is a smoothing operator, with Schwartz kernel 
$\Pi^\mu_k(\cdot,\cdot)\in \mathcal{C}^\infty(X\times X)$  given by
\begin{equation}
 \label{eqn:kernel-szego-equiv}
\Pi^\mu_k(x,y)=\sum_j\,s_j^{(k)}(x)\cdot \overline{s_j^{(k)}(y)}\,\,\,\,\,\,\,\,\,(x,y\in X)
\end{equation}
for any choice of an orthonormal basis $\left(s_j^{(k)}\right)$ of $H^\mu_k(X)$.
The diagonal restriction $x\mapsto \Pi^\mu_k(x,x)$ descends to a well-defined $\mathcal{C}^\infty$
function on $M$.

Also, if $\Phi>0$ then $\xi_X(x)\neq 0$ for every $x\in X$ by (\ref{eqn:infinitesimal-lift}); hence
$\mu^X$ is locally free, and every $x\in X$ has finite stabilizer $T_x\subseteq \mathbf{T}^1$. As $\mu^X$
commutes with scalar multiplication, $T_x$ only depends on $m=\pi(x)\in M$; we shall emphasize this
by denoting $T_x$ by $T_m$. For instance, $T_m=\{1\}$ for every $m\in M$ if $\mu^M$ is trivial and $\Phi=1$.
While $T_m$ is generally not constant on $M$, it equals a fixed finite subgroup
$T_\mathrm{gen}\subseteq \mathbf{T}^1$ on a dense open subset $M'\subseteq M$.
Then $T_\mathrm{gen}$ stabilizes every $x\in X$; after passing to the quotient,
we may reduce to the case $T_\mathrm{gen}=\{1\}$.
By Corollary 1.1 of \cite{pao-IJM}, at a point $x\in X$ where $T_m$ is trivial
$\Pi^\mu_k(x,x)$ satisfies an asymptotic expansion as $k\rightarrow +\infty$ of the form
\begin{equation}
 \label{eqn:asymptotic-expansion-szego-special-case}
\Pi^\mu_k(x,x)\sim \left(\frac k\pi\right)^d\,\sum_{j\ge 0}k^{-j}\,S_j^\mu(m),
\end{equation}
where $S_0^\mu(m)=\Phi(m)^{-(d+1)}$. Here we shall focus on the lower order terms 
$S_j^\mu$.

More generally, given a real $f\in \mathcal{C}^\infty(M)$, one can consider
the associated Toeplitz operators
$T^\mu_k[f]=:\Pi^\mu_k\circ M_f\circ \Pi^\mu_k$, viewed as self-adjoint endomorphisms 
of $H^\mu_k(X)$; here $M_f:L^2(X)\rightarrow L^2(X)$ is multiplication by $f\circ \pi$.
Assuming $\Phi>0$, this is also a smoothing operator, whose distributional kernel may be expressed as
\begin{eqnarray}
 \label{eqn:kernel-topelitz-equiv}
T^\mu_k[f](x,x')&=&\int_X\Pi^\mu_k(x,y)\,f(y)\,\Pi^\mu_k(y,x')\,dV_X(y)\nonumber\\
&=&\sum_j T^\mu_k[f]\big(s_j^{(k)}\big)(x)\cdot \overline{s_j^{(k)}(y)}\,\,\,\,\,\,\,\,\,\,\,(x,y\in X),
\end{eqnarray}
where we write $f(y)$ for $f\big(\pi(y)\big)$.
The diagonal restriction $x\mapsto T^\mu_k[f](x,x)$ also descends to $M$.
We shall see that $T^\mu_k[f](x,x')$ has near diagonal scaling asymptotics (that is, for 
$x\rightarrow x'$) analogous to
those of $\Pi^\mu_k$ in Theorem 1 of \cite{pao-IJM}, and investigate the lower order terms
in the asymptotics of the diagonal restriction $T^\mu_k[f](x,x)$. We shall then derive from this 
an asymptotic expansion for an \lq equivariant Berezin transform\rq\,, and consider the 
relation between commutators of Toeplitz operators and Poisson brackets of the corresponding Hamiltonians.
Before describing our results in detail, we need to specify the geometric setting somewhat.

We shall assume without loss that $T_\mathrm{gen}$ 
is trivial; then $\mu^X$ is free 
on a dense $\nu^X\times \mu^X$-invariant open subset $X'\subseteq X$ 
(since $\nu^X$ - given by (\ref{eqn:action-S-X}) - and $\mu^X$ commute, 
we may consider the product action). Thus $M'=:\pi(X')\subseteq M$
is also open and dense.

The quotient $N=X/\mathbf{T}^1$ is an orbifold, and
the dense open subset 
$N'=:X'/\mathbf{T}^1\subseteq N$ is a manifold; 
the restricted projection $\kappa:X'\rightarrow N'$ is a circle bundle, and passing from
$\pi$ to $\kappa$ the roles of $\mu^X$ and $\nu^X$ get interchanged.

More precisely, $\beta=:\alpha/\Phi$ is a connection 1-form for 
$\kappa$, defining the same horizontal distribution as $\alpha$, 
and there is on $N'$ a naturally induced K\"{a}hler structure
$(N',I,\eta)$ with
$d\beta=2\,\kappa^*(\eta)$, and if $\omega$ is real-analytic then so is $\eta$.
Furthermore, $\nu^X$ descends to an action $\nu^N:\mathbf{T}^1\times N'\rightarrow N'$, which turns out
to be holomorphic with respect to $I$ and Hamiltonian with respect to $2\,\eta$. If as generating 
Hamiltonian for $\nu^N$ we choose $\Phi^{-1}$, descended to a function on $N'$, $\nu^X$ is the 
corresponding contact lift of $\nu^N$ to $(X',\beta)$ in the sense of (\ref{eqn:infinitesimal-lift}).

Every $\mu^M$-invariant $\mathcal{C}^\infty$ function $f=f(m)$ on $M$
lifts to $\nu^X\times \mu^X$-invariant function $f=f(x)$ on $X$, and then descends to a $\nu^N$-invariant
$\mathcal{C}^\infty$ function $f=f(n)$ on $N'$. In the reverse direction, a $\mathcal{C}^\infty$ $\nu^N$-invariant
function $f=f(n)$ on $N'$ yields a $\mu^M$-invariant $\mathcal{C}^\infty$ function $f=f(m)$ on $M'$. 
We thus have a natural algebraic isomorphism
between spaces of invariant smooth functions:
$$
\mathcal{C}^\infty(M')^{\mu}\cong \mathcal{C}^\infty(N')^{\nu}.
$$
If $\omega$ is real-analytic, this restricts to an isomorphism between the
corresponding subspaces of invariant real-analytic functions:
$$
\mathcal{C}^\varpi (M')^{\mu}\cong \mathcal{C}^\varpi (N')^{\nu}.
$$

With this understanding,
we shall think of $\Phi$ as being defined on $M$, $X$, or $N$ according to the context, 
and drop the symbols of pull-back
or push-forward. Similarly, let $\varrho_N$ be the scalar curvature of the
K\"{a}hler structure $(N',I,2\,\eta)$; 
then $\varrho_N$ is $\nu^N$-invariant, and may be viewed as a $\mu^M$-invariant function on $M'$. 
By the same principle,
the Laplace-Beltrami operator $\Delta_N$ of $(N',I,2\,\eta)$
acts on $\mu^M$-invariant 
functions on $M'$ (see \S \ref{subsctn:kahler-stuff} for precise definitions).

An important ingredient of the present analysis is the study by
Engli\v{s}
of the asymptotics of Laplace integrals
on a real-analytic K\"{a}hler manifold. Namely, let $(g_{k\overline{l}})$ be a real-analytic K\"{a}hler
metric on an open subset $U\subseteq \mathbf{C}^d$, and suppose that $\Xi$ is a K\"{a}hler potential
for $(g_{k\overline{l}})$ on $U$. Let $\widetilde{\Xi}$ be a sesqui-holomorphic extension of
$\Xi$ to some open neighborhood  
$\widehat{U}\subseteq U\times U$ of the diagonal. Calabi's \textit{diastasis function}
is given by
\begin{equation}
\label{eqn:calabi-diastasis}
 \mathcal{D}(z,w)=:\Xi(z)+\Xi(w)-\widetilde{\Xi}(z,w)-\widetilde{\Xi}(w,z)\,\,\,\,\,\,\,
\,\,\left((z,w)\in  \widehat{U}\right);
\end{equation}
it is an intrinsic attribute of $(g_{k\overline{l}})$, that is, it does not depend on the choice of $\Xi$,
and it satisfies $\mathcal{D}(z,z)=0$ and $\mathcal{D}(z,w)>0$ if $z\neq w$  \cite{calabi}
(see also the discussions in \cite{cahen-gutt-rawnsley-II} and \cite{loi-2}).

In \cite{englis}, Engli\v{s} considers the asymptotics as $\lambda\rightarrow +\infty$
of integrals of the form
\begin{equation}
 \label{eqn:laplace-integral}
I(\lambda,y)=:\int_U e^{-\lambda\,\mathcal{D}(x,y)}\,f(x)\,g(x)\,dx,
\end{equation}
where $g=:\det [g_{k\overline{l}}]$ and $dx$ denotes the Lebesgue measure on $\mathbb{C}^d$. By Theorem 3 of
\cite{englis}, there is an asymptotic expansion of the form
\begin{equation}
 \label{eqn:englis-asymptotic-expansion}
I(\lambda,y)\sim \left(\frac \pi \lambda\right)^d\,\sum_{j\ge 0}\lambda^{-j}\,\left.R_j^U(f)\right|_y,
\end{equation}
where the $R_j^U$'s are covariant differential operators, that may be expressed in a universal manner
in terms of the metric, the
curvature tensor, and their covariant derivatives; in particular, $R_0=\mathrm{id}$ and
$R_1=\Delta_N-\varrho_N/2$ (the opposite sign convention is used in \cite{englis}
for the curvature tensor and for $\varrho_N$).
Engli\v{s} also provided an explicit description of $R_j^U$ for $j\le 3$; the higher $R_j^U$'s and their differential
geometric significance were further investigated in \cite{loi-2}, and a graph theoretic forula for them
was given in \cite{hao-xu}. Because $\mathcal{D}$ and 
the $R_l^U$'s are intrinsically defined, the expansion (\ref{eqn:englis-asymptotic-expansion})
holds globally on any real-analytic K\"{a}hler manifold $(S,g)$, in which case we shall denote
the covariant operators by $R_j^S$.

\begin{thm}
 \label{thm:lower-order-terms}
With the notation above, suppose that $\omega$ is real-analytic,
$\Phi>0$ and $T_{\mathrm{gen}}$ is trivial. 
Then the invariant functions 
$S_j^\mu:M'\rightarrow \mathbb{R}$ 
in (\ref{eqn:asymptotic-expansion-szego-special-case})
are determined as follows. First, $S_0^\mu=\Phi^{-(d+1)}$. Next, for some
$j\ge 0$ suppose inductively that 
$$
S_0^\mu,\ldots,S_j^\mu\in 
\mathcal{C}^\varpi(M')^\mu\cong \mathcal{C}^\varpi(N')^\nu 
$$ 
have been constructed, and let $\widetilde{S_0^\mu},\ldots,\widetilde{S_j^\mu}$
be their respective sesquiholomorphic extensions as elements of $\mathcal{C}^\varpi(N')^\nu$.
Define 
\begin{equation}
 \label{eqn:Z-j-general}
Z_j(n_0,n)=:\Phi(n)^{d+1}\,\sum_{a+b=j}\widetilde{S_a^\mu}(n_0,n)\,\widetilde{S_b^\mu}(n,n_0).
\end{equation}
Then, thinking of the $R_r^N$'s as acting on the variable $n$ and of $n_0$ as a parameter,
\begin{eqnarray}
 \label{eqn:S-j-general}
S_{j+1}^\mu(n_0)&=&-\Phi(n_0)^{d+1}\,\sum_{l=1}^j S^\mu_l(n_0)\,S^\mu_{j+1-l}(n_0)\nonumber\\
&&-\sum_{r=1}^{j+1}\left.R_r^N\big(Z_{j+1-r}(n_0,\cdot)\right)|_{n=n_0}.
\end{eqnarray}
\end{thm}

Since the $R_r^N$'s are universal intrinsic attributes of the K\"{a}hler
manifold $(N,K,\eta)$,
(\ref{eqn:S-j-general}) expresses the $S_j^\mu$'s as a universal intrinsic attribute
of the Hamiltonian action, through the geometry of its quotient. 
As mentioned, the $R_r^N$'s were computed in \S 4 of \cite{englis}, in \cite{loi-2} and \cite{hao-xu};
thus, in principle,
(\ref{eqn:S-j-general}) determines $S^\mu_l$ explicitly 
in terms of the geometry of the quotient $N'$. 
Let us consider $S^\mu_1$:

\begin{cor}
 \label{cor:lower-order-terms}
Under the assumptions of Theorem \ref{thm:lower-order-terms}, we have
\begin{eqnarray*}
S_1^\mu&=&\frac 12\,\varrho_N\,\Phi^{-(d+1)}\\
&&+(d+1)\,\Phi(n_0)^{-(d+2)}\,\left[
\dfrac{1}{2\,\Phi}\,\big\|\mathrm{grad}_N(\Phi)\big\|^2
-\Delta_N(\Phi)
\right].\nonumber
\end{eqnarray*}
\end{cor}

Here $\varrho_N$, the gradient $\mathrm{grad}_N\Phi$ of $\Phi$ as a function on $N'$, 
and the Laplacian $\Delta_N(\Phi)$ are
taken with respect to the
K\"{a}hler structure $(N,I,2\eta)$, and $\|\cdot\|_N$ is the norm in the same metric.
Their relation to the corresponding objects on $M$ is explained in \S \ref{sctn:kahler-structure-N'}
and \S \ref{sct:laplacian-invariant-functions} (see (\ref{eqn:square-norm-gradient-MN}) and 
(\ref{eqn:laplacian-comparison-final-global})). If $\Phi=1$, we recover Lu's
subprincipal term \cite{lu}.

\begin{rem}
 A notational remark is in order. If, working in a system of local holomorphic
coordinates, $\gamma_{a\overline{b}}$ is a K\"{a}hler form, the
corresponding K\"{a}hler metric here is $\rho_{a\overline{b}}=-i\,\gamma_{a\overline{b}}$
(see the discussion in \S \ref{subsctn:kahler-stuff} and (\ref{eqn:espressione-locale-gamma})).
In the literature, often a factor $1/2$ (or $1/(2\pi)$) 
is included on the left hand side
of the previous relation; with this convention, 
the previous invariants would be associated to $(N,I,\eta)$ (\cite{tian-book},
\cite{lu}).
\end{rem}

Next let us dwell on the local asymptotics of the Toeplitz kernels
$T^\mu_k[f](\cdot,\cdot)$.
Firstly, by Theorem 1 of \cite{pao-IJM} we have
$\Pi^\mu_k(x',x'')=O\left(k^{-\infty}\right)$ uniformly for
$\mathrm{dist}_X\left(\mathbf{T}^1\cdot x',x''\right)\ge C\,k^{\epsilon-1/2}$, for any given
$\epsilon>0$. In view of (\ref{eqn:kernel-topelitz-equiv}), the same holds of $T^\mu_k[f]$.
We can then focus on the local asymptotics of $T^\mu_k[f](x',x'')$ for $x''\rightarrow
\mathbf{T}^1\cdot x'$.
In view of (\ref{eqn:isotypes}) and (\ref{eqn:kernel-szego-equiv}), 
for any $e^{i\vartheta}\in \mathbf{T}^1$ we have
\begin{equation}
 \label{eqn:equivariance-property}
T^\mu_k[f]\left(\mu^X_{e^{-i\vartheta}}(x'),x''\right)=
e^{ik\vartheta}\,T^\mu_k[f]\left(x',x''\right)=
T^\mu_k[f]\left(x',\mu^X_{e^{i\vartheta}}(x'')\right).
\end{equation}
Therefore, we need only consider the asymptotics of $T^\mu_k[f](x',x'')$ for $x''\rightarrow
x'$. Predictably, these exhibit the same kind of scaling behavior as the asymptotics of
$\Pi^\mu_k(x',x'')$ for $x'\rightarrow x''$ (Theorem 2 of \cite{pao-IJM}).

This is best expressed in terms of Heisenberg local coordinates 
(in the following: HLC for short)
$x+(\theta,\mathbf{v})$ centered at
a given $x\in X$; here $(\theta,\mathbf{v})\in (-\pi,\pi)\times B_{2d}(\mathbf{0},\delta)$,
where $B_{2d}(\mathbf{0},\delta)\subseteq \mathbb{C}^d$ is 
the open ball centered at the origin and of radius $\delta>0$.
It is in these coordinates that the near-diagonal scaling asymptotics 
of the standard equivariant Szeg\"{o}
kernels $\Pi_k$ exhibit their universal nature \cite{bsz}, \cite{sz}, 
and by \cite{pao-IJM} the same holds of the
$\Pi^\mu_k$'s. While we refer to \cite{sz} for a precise definition, let us recall that
Heisenberg local coordinates enjoy the following properties.

Firstly, the parametrized submanifold $\gamma_x:\mathbf{v}\mapsto x+(0,\mathbf{v})$ is horizontal,
that is, tangent to $\ker(\alpha)\subseteq TX$, at $\mathbf{v}=\mathbf{0}$. 
In view of (\ref{eqn:infinitesimal-lift}), and given that $\Phi>0$, 
$\gamma_x$ is transverse to the $\mu^X$-orbit $\mathbf{T}^1\cdot x$;
hence for $\mathbf{v}\sim \mathbf{0}$ we have
\begin{equation}
 \label{eqn:bound-orbit-distance-HLC-first}
D_2\,\|\mathbf{v}\|\ge
\mathrm{dist}_X\left(\mathbf{T}^1\cdot x,x+\mathbf{v}\right)\ge D_1\,\|\mathbf{v}\|,
\end{equation}
for some fixed $D_1,\,D_2>0$.

Since HLC centered at $x\in X$ come with a built-in unitary isomorphism
$T_mM\cong \mathbb{C}^d$, where $m=\pi(x)\in X$, we may use the expression $x+(\theta,\mathbf{v})$
when $\mathbf{v}\in T_mM$ has sufficiently small norm.

Finally, scalar multiplication by $e^{i\vartheta}\in \mathbf{T}^1$ is expressed
in HLC by a translation in the angular coordinate: where both terms are defined, we have
\begin{equation}
 \label{eqn:HLC-scalar-multiplication}
e^{i\vartheta}\cdot \big(x+(\theta,\mathbf{v})\big)=x+(\vartheta+\theta,\mathbf{v}).
\end{equation}
We shall set $x+\mathbf{v}=:x+(0,\mathbf{v})$. 

Given (\ref{eqn:equivariance-property})
and the previous transversality argument, we need only consider the asymptotics
of $T^\mu_k[f](x+\mathbf{v},x+\mathbf{w})$ for $\mathbf{v},\,\mathbf{w}\rightarrow 0$.
Following \cite{sz}, let us define, for $\mathbf{v},\mathbf{w}\in T_mM$,
\begin{equation}
 \label{eqn:definition-psi-2}
\psi_2(\mathbf{v},\mathbf{w})=:-i\,\omega_m(\mathbf{v},\mathbf{w})-\frac 12\,
\|\mathbf{v}-\mathbf{w}\|_m^2,
\end{equation}
where $\|\cdot\|_m$ is the Euclidean norm on the unitary vector space
$(T_mM,\omega_m,J_m)$.

\begin{thm}
 \label{thm:toeplitz-scaling}
Assume as above that $\Phi>0$. Then for any $f\in \mathcal{C}^\infty(M)^\mu$ we have
\begin{enumerate}
 \item $T^\mu_k[f]=0$ for any $k\le 0$.
\item For any $C,\epsilon>0$, we have
$T^\mu_k[f]\big(x',x''\big)=O\left(k^{-\infty}\right)$ as $k\rightarrow +\infty$,
uniformly for $\mathrm{dist}_X\left(\mathbf{T}^1\cdot x',x''\right)\ge C\,k^{\epsilon-1/2}$.
\item Suppose $x\in X$ and fix a system of HLC on $X$ centered at $x$. Set
$m=:\pi(x)$. Then uniformly for $\mathbf{v},\,\mathbf{w}\in T_mM$ with 
$\|\mathbf{v}\|, \,\|\mathbf{w}\|\le C\,k^{1/9}$, as $k\rightarrow +\infty$
we have an asymptotic expansion of the form
\begin{eqnarray*}
 \lefteqn{
T^\mu_k[f]\left(x+\dfrac{\mathbf{v}}{\sqrt{k}},x+\dfrac{\mathbf{w}}{\sqrt{k}}\right)
}\\
&=&\left(\frac k\pi\right)^d 
\sum _{t\in T_m}t^k\,e^{\psi_2\left(d_m\mu^M_{t^{-1}}(\mathbf{v}),\mathbf{w}\right)
/\Phi(m)}\cdot A_t(m,\mathbf{v},\mathbf{w}),
\end{eqnarray*}
with
\begin{eqnarray*}
A_t(m,\mathbf{v},\mathbf{w},f)&\sim& 
\sum _{j\ge 0}k^{-j/2}\,R_j\left(m,d_m\mu^M_{t^{-1}}(\mathbf{v}),\mathbf{w},f\right),
\end{eqnarray*}
where the $R_j(\cdot,\cdot,\cdot,\cdot)$'s are polynomial in $\mathbf{v}$ and $\mathbf{w}$
and differential operators in $f$.
In particular,
$$
R_0\left(m,d_m\mu^M_{t^{-1}}(\mathbf{v}),\mathbf{w},f\right)=\Phi(m)^{-(d+1)}\,f(m).
$$

\item The previous asymptotic expansion goes down by integer steps
when $\mathbf{v}=\mathbf{w}=0$ (that is, only powers of $k^{-1}$ appear in the diagonal
asymptotics).
\end{enumerate}

 \end{thm}

Theorem \ref{thm:toeplitz-scaling} might be proven by a microlocal argument along the lines of
the one used for Theorem 1 of \cite{pao-IJM}; to avoid introducing too much machinery, we shall
instead deduce it as a consequence of Theorem 1 of \cite{pao-IJM}, by inserting in (\ref{eqn:kernel-topelitz-equiv})
the near-diagonal
scaling asymptotics for $\Pi^\mu_k$.

\begin{cor}
\label{cor:toeplitz-scaling}
In the situation of Theorem \ref{thm:toeplitz-scaling},
suppose in addition that $T_{\mathrm{gen}}$ is trivial. If $x\in X'$, 
then as $k\rightarrow+\infty$
there is an asymptotic expansion
\begin{eqnarray*}
 T^\mu_k[f]\left(x,x\right)&\sim&
\left(\frac k\pi\right)^d 
\sum _{j\ge 0}k^{-j}\,S_j^\mu[f]\left(m\right),
\end{eqnarray*}
where $m=\pi(x)$ and every $S_j^\mu[f]\in \mathcal{C}^\infty(M')^{\mu}$.
In particular, we have
$$
S_0^\mu[f]=\Phi^{-(d+1)}\cdot f.
$$
\end{cor}

When $\Phi=1$, corresponding results were obtained in Lemma 4.6 of \cite{mm00} and
Lemma 7.2.4 of \cite{mm0}, covering the case of symplectic manifolds in the presence of a twisting vector
bundle.

Let us consider the lower order $f_j^\mu$'s.

\begin{thm}
 \label{thm:toeplitz-lower-order-terms}
Under the assumptions of Corollary \ref{cor:toeplitz-scaling},
assume also that $\omega$ is real-analytic. Then
for every $j=0,1,2,\ldots$ 
we have $S_j^\mu[f]=P_j^\mu(f)$, where each $P^\mu_j$ is a differential operator
of degree $\le 2j$. More precisely, viewed as a $\nu^N$-invariant
function on $N$, $f_j$ is given by
$$
S_j^\mu[f](n_0)=P^\mu_j(f)(n_0)
=\sum_{r+s=j}\left.R_r^N\big(f(\cdot )\,Z_s(n_0,\cdot)\big)\right|_{n=n_0}.
$$
\end{thm}

\begin{rem}
 Clearly, $S_j^\mu=S_j^\mu[1]$ for every $j\ge 0$.
\end{rem}

\begin{cor}
\label{cor:toeplitz-lower-order-terms} 
In the situation of Theorem \ref{thm:toeplitz-lower-order-terms},
$$
S_1^\mu[f]=\Phi^{-(\mathrm{d}+1)}\,\Delta_N(f)+S^\mu_1\cdot f.
$$
\end{cor}

For $\Phi=1$, the corresponding result to Corollary \ref{cor:toeplitz-lower-order-terms} 
was obtained in (0.13) of \cite{mm2}.

For a general discussion of the Berezin transform in the K\"{a}hler context, we refer, say, to
\cite{ae}, \cite{cahen-gutt-rawnsley-I}, \cite{englis}, \cite{loi-3}, \cite{schlichenmaier}.
Here we adopt the following natural adjustment.

\begin{defn}\label{defn:equivariant-berezin-transform}
 If $f\in \mathcal{C}^\infty(M)$ and $k=0,1,2,\ldots$, let the $k$-th
\textit{$\mu$-equivariant Berezin transform} of $f$ be given by
$$
\mathrm{Ber}^\mu_k[f](m)=:\dfrac{T^\mu_k[f](x,x)}{\Pi^\mu_k(x,x)}\,\,\,\,\,\,\,\,\,\,\,\,\,\,(m\in M)
$$
for any choice of $x\in \pi^{-1}(m)$.
\end{defn}

\begin{cor}
 \label{cor:berezin-transform}
Assume that $\omega$ is real-analytic,
$\Phi>0$ and $T_{\mathrm{gen}}=\{1\}$. If $f\in \mathcal{C}^\infty(M)^{\mu}$, then 
as $k\rightarrow +\infty$ on $M'$, uniformly on compact subsets of $M'$, 
there is an asymptotic expansion of the form
\begin{eqnarray*}
 \mathrm{Ber}^\mu_k[f]&\sim&\sum_{j\ge 0}k^{-j}\,B_j^\mu(f),
\end{eqnarray*}
where every $B_j^\mu$ is a differential operator of degree $2j$. In particular,
$B_0^\mu=\mathrm{id}$ and $B_1^\mu=\Delta_N$.
\end{cor}

A corresponding result for $\Phi=1$ was given in
\cite{englis}.

The following analogue of the Heisenberg correspondence 
relates the commutator of two equivariant Toeplitz operators
to the Poisson brackets of the corresponding Hamiltonians. Let $\{\cdot,\cdot\}_M$
and $\{\cdot,\cdot\}_N$ denote, respectively, Poisson brackets on $(M,2\,\omega)$
and $(N',2\,\eta)$. By restriction, they yield maps 
$$
\{\cdot,\cdot\}_M,\,\{\cdot,\cdot\}_N:\mathcal{C}^\infty(M')^{\mu}\times \mathcal{C}^\infty(M')^{\mu}
\rightarrow \mathcal{C}^\infty(M')^{\mu}.
$$

\begin{thm}
 \label{thm:heisenberg-relation}
Assume that  $\omega$ is real-analytic,
$\Phi>0$ and $T_{\mathrm{gen}}=\{1\}$. 
Let $f,\,g\in \mathcal{C}^\infty(M)^{\mu}$ be real-valued,
and denote by $E^\mu_k[f,g](\cdot,\cdot)\in \mathcal{C}^\infty(X\times X)$ the Schwartz kernel
of the composition $T^\mu_k[f]\circ T^\mu_k[g]$. Then
uniformly on compact subsets of $M'$ as $k\rightarrow+\infty$ we have
\begin{eqnarray*}
\lefteqn{ E^\mu_k[f,g](x,x)-E^\mu_k[f,g](x,x)}\\
&=&\left(\frac{k}{\pi}\right)^d\,
\left[-\dfrac ik\,\Phi(m)^{-(\mathrm{d}+1)}\,\big\{f,g\big\}_N(m)+O\left(k^{-2}\right)\right]\nonumber\\
&=&\left(\frac{k}{\pi}\right)^d\,
\left[-\dfrac ik\,\Phi(m)^{-\mathrm{d}}\,\big\{f,g\big\}_M(m)+O\left(k^{-2}\right)\right],
\end{eqnarray*}
for any $x\in \pi^{-1}(m)$.
\end{thm}

In the course of the proof, one actually establishes an asymptotic expansion for $E^\mu_k[f,g](x,x)$ (see 
(\ref{eqn:composition-toeplitz-parametrized-3})):
\begin{equation}
 E^\mu_k[f,g](x,x)\sim\left(\frac k\pi\right)^{d}\,\sum_jk^{-j}\,A_j[f,g](x),
\end{equation}
where $A_0[f,g]=\Phi^{-(d+1)}\cdot f\,g$ and
$$
A_1[f,g]=\Phi^{-(d+1)}\,\left[f\,\Delta_Ng+g\,\Delta_Nf+ \left\langle \mathrm{grad}_N(f)^{(0,1)} ,\mathrm{grad}_N(g)^{(1,0)}\right\rangle\right]+
S_1^\mu\cdot f\,g
$$
(we leave the explicit computation to the reader). 
When $\Phi=1$, the formula for $A_1[f,g]$ was obtained in (0.16) of \cite{mm2}.

As explained in the references above for the standard case,
this expansion can be used to define in a natural manner 
a $*$-product on $\mathcal{C}^\infty(M')^\mu$ (depending on $\Phi$), but we won't
discuss this here. 

\section{Preliminaries}

\subsection{Some notation and recalls from K\"{a}hler geometry}
\label{subsctn:kahler-stuff}

Let $(P,K)$ be a $d$-dimensional complex manifold and let 
$(P,K,\gamma)$ be a K\"{a}hler structure on it, with associated
covariant metric
tensor $\rho(\cdot,\cdot)=:\gamma\big(\cdot,K(\cdot)\big)$; also,
let $\ell=:\rho-i\,\gamma$ be the associated Hermitian metric.
Given holomorphic local coordinates $(z_a)$ on $P$, we shall let
$\partial_a=:\partial/\partial z_a$ and 
$\partial_{\overline{a}}=:\partial/\partial \overline{z}_a$,
$\rho_{a\overline{b}}=:\rho(\partial_a,\partial_{\overline{b}})$, 
$\gamma_{a\overline{b}}=:\gamma(\partial_a,\partial_{\overline{b}})$.
$\ell_{a\overline{b}}=:\ell(\partial_a,\partial_{\overline{b}})$.
Then locally
\begin{equation}
 \label{eqn:espressione-locale-gamma}
\gamma=\sum_{a,b}\gamma_{a\overline{b}}\,dz_a\wedge d\overline{z}_b=
i\,\sum_{a,b}\rho_{a\overline{b}}\,dz_a\wedge d\overline{z}_b=
\frac i2\,\sum_{a,b}\ell_{a\overline{b}}\,dz_a\wedge d\overline{z}_b.
\end{equation}

Consider the real local frame 
$\mathcal{B}=\big(\partial/\partial x_1,\ldots,\partial/\partial x_d,
\partial/\partial y_1,\ldots,\partial/\partial y_d\big)$, where
$z_j=x_j+i\,y_j$ is the decomposition in real and imaginary parts, and
denote by $M_{\mathcal{B}}(\rho)$ the matrix representing $\rho$ in this
frame. Then
$$
\det M_{\mathcal{B}}(\rho)=4^d\,\det\big([\rho_{a\overline{b}}]\big)^2.
$$

Therefore, the Riemannian volume form of $(P,\rho)$ is
\begin{eqnarray}
 \label{eqn:riem-vol-element}
dV_P=\dfrac{1}{d!}\,\gamma^{\wedge d}&=&\sqrt{\det \left(M_{\mathcal{B}}(\rho)\right)}\cdot dx_1\wedge \cdots dx_d\wedge dy_1\cdots \wedge
dy_d\nonumber\\
&=&2^d\,\det \big([\rho_{k\overline{l}}]\big)\cdot dx_1\wedge \cdots dx_d\wedge dy_1\cdots \wedge
dy_d\nonumber\\
&=&\det \big([2\,\rho_{k\overline{l}}]\big)\cdot dx_1\wedge \cdots dx_d\wedge dy_1\cdots \wedge
dy_d.
\end{eqnarray}

Let $R$ be the covariant curvature tensor of the Riemannian manifold
$(P,\rho)$, with components
$R_{a\overline{b}c\overline{d}}=R\big(\partial_a,\partial_{\overline{b}},
\partial_c,\partial_{\overline{d}}\big)$ \cite{tian-book}. 

We shall set (leaving the metric understood
and adopting Einstein notation)
\begin{equation}
 \label{eqn:scalar-curvature}
\varrho_P=:\rho^{\overline{b}a}\,\rho^{\overline{d}c}\,R_{a\overline{b}c\overline{d}};
\end{equation}
this is $1/2$ of the ordinary Riemannian scalar curvature $\mathrm{scal}_P$.

Simlarly, for $f\in \mathcal{C}^\infty$, we shall let
\begin{equation}
 \label{eqn:laplace-beltrami}
\Delta_P(f)=:\rho^{\overline{b}a}\,\partial_a\partial_{\overline{b}}f,
\end{equation}
which is $1/2$ times the ordinary Riemannian Laplace-Beltrami operator.

The gradient of $f$ is locally given by 
\begin{equation}
 \label{eqn:gradient-local}
\mathrm{grad}_P(f)=\rho^{\overline{b}a}\,(\partial_{\overline{b}}f)\,\partial_a+
\rho^{\overline{b}a}\,(\partial_af)\,\partial_{\overline{b}},
\end{equation}
and its square norm is given by
\begin{equation}
 \label{eqn:square-norm-gradient}
\big\|\mathrm{grad}_P(f)\big\|^2=
2\,\rho^{\overline{b}a}\,(\partial_af)\,(\partial_{\overline{b}}f).
\end{equation}

Since $\Delta_P$ here is $1/2$ times the
ordinary Laplace-Beltrami operator, we have for any $f_1,\,f_2\in \mathcal{C}^\infty(P)$:
\begin{equation*}
 \Delta_P(f_1\cdot f_2)=f_1\,\Delta_P(f_2)+\rho\big(\mathrm{grad}_P(f_1),\mathrm{grad}_P(f_2)\big)
+f_2\,\Delta_P(f_1).
\end{equation*}
It follows inductively that for any $f\in \mathcal{C}^\infty(P)$ and $l\ge 0$ we have
\begin{equation}
 \label{eqn:laplacian-powers-inductive}
\Delta_P\left(f^l\right)=l\,f^{l-1}\,\Delta_P(f)
+\dfrac{(l-1)\,l}{2}\,f^{l-2}\,\big\|\mathrm{grad}_P(f)\big\|^2.
\end{equation}

Let us now consider the Poisson brackets $\{f,g\}_P=\gamma (H_f,H_g)$
of two real functions
$f,\,g\in \mathcal{C}^\infty(P)$ in the symplectic structure $(P,\gamma)$;
here $H_f$ is the Hamiltonian vector field of $f$ with respect to $\gamma$.
We have $H_f=-K\big(\mathrm{grad}_P(f)\big)$, hence
given (\ref{eqn:gradient-local})
\begin{eqnarray}
 \label{eqn:poisson-bracket-kahler}
\{f,g\}_P&=&\gamma\Big(K\big(\mathrm{grad}_P(f)\big),K\big(\mathrm{grad}_P(g)\big)\Big)=
\gamma\big(\mathrm{grad}_P(f),\mathrm{grad}_P(g)\big)\nonumber\\
&=&\rho\Big(K\big(\mathrm{grad}_P(f)\big),\mathrm{grad}_P(g)\Big)\nonumber\\
&=&i\,\rho\Big(\rho^{\overline{b}a}\,(\partial_{\overline{b}}f)\ ,\partial_a-
\rho^{\overline{b}a}\,(\partial_af)\,\partial_{\overline{b}},
\rho^{\overline{d}c}\,(\partial_{\overline{d}}g)\,\partial_c+
\rho^{\overline{d}c}\,(\partial_c g)\,\partial_{\overline{d}}
\Big)\nonumber\\
&=&\frac 1i\,\rho^{\overline{d}c}\,\Big[(\partial_cf)\,(\partial_{\overline{d}}g)
-(\partial_c g)\,(\partial_{\overline{d}}f)\Big].
\end{eqnarray}

\begin{lem}
 \label{lem:real-analytic-moment-map}
Let $(P,K,\gamma)$ be a K\"{a}hler manifold, with
$\gamma$ real-analytic. Let $\varPhi:P\rightarrow\mathbb{R}$
be a real $\mathcal{C}^{\infty}(M)$ function whose Hamiltonian flow
with respect to $\gamma$
is holomorphic with respect to $K$. Then $\varPhi$ is real-analytic.
\end{lem}

\begin{proof}
Let
$T^cP=TP\otimes \mathbb{C}$ be the complexified tangent bundle
of $P$, and $T^cP=T'P\oplus T''P$ its decomposition into
$\pm i$-eigenbundles of $K$.
Let $\upsilon_\varPhi\in \mathfrak{X}(P)$
be the Hamiltonian vector field of $\varPhi$
with respect to $\gamma$. If $ \upsilon_\varPhi=\upsilon_\varPhi'
\oplus \upsilon_\varPhi''$, with $\upsilon_\varPhi'\in T'P$ and
$\upsilon_\varPhi''=\overline{\upsilon_\varPhi'}\in T''P$, 
then $\upsilon_\varPhi'$ is holomorphic, whence real-analytic.
Then clearly $\upsilon_\varPhi$ is real-analytic as well, and therefore
so is its differential $d\upsilon_\varPhi=\iota(\upsilon_\varPhi)\,\gamma$.
This forces $\varPhi$ itself to be real-analytic (say by Proposition 2.2.10
of \cite{parks-krantz}).
\end{proof}

\subsubsection{The Laplacian and sesquiholomorphic extensions}

We give here a couple of technical Lemmas that will be handy in the proof of Corollary
\ref{cor:lower-order-terms}.

\begin{lem}
\label{lem:technical}
Let $(P,K,\gamma)$ be a K\"{a}hler manifold,
and consider $f\in \mathcal{C}^\varpi(P)$ with $f>0$. Let $\widetilde{f}(\cdot,\cdot)$
be the sesquiholomorphic extension of $f$ to an open neighborhood 
$\widetilde{P}\subseteq P\times P$ of the diagonal
(thus $\widetilde{f}(\cdot,\cdot)$ is holomorphic in the first entry and anti-holomorphic
in the second, and $\widetilde{f}(p,p)=f(p)$ for all $p\in P$). 
Given $p_0\in P$, let 
$P'\subseteq P$ be an open neighborhood  of $p_0$ so small that
$P'\times P'\subseteq \widetilde{P}$ and $\widetilde{f}(p_0,p)\neq 0$
for all $p\in P'$. Define 
$f_1,\,f_2,\,F_f\in \mathcal{C}^\varpi(P')$ by setting:
$$
f_1(p)=:f(p_0,p),\,\,\,\,\,f_2(p)=:f(p,p_0)=\overline{f_1(p)},\,\,\,\,\,
F_f(p)=\dfrac{f(p)}{f_1(p)\,f_2(p)}\,\,\,\,\,\,(p\in P').
$$ 
Thus $f_1$ is anti-holomorphic, $f_2$ is holomorphic, and
$F_f>0$. Then
\begin{equation}
 \label{eqn:laplacian-sesquiholomorphic}
\Delta_P(F_f)(p_0)=\dfrac{1}{f(p_0)^2}\,\left[\Delta_P(f)(p_0)
-\dfrac{1}{2\,f(p_0)}\,\big\|\mathrm{grad}_P(f)(p_0)\big\|^2\right],
\end{equation}
where the terms involved are given by (\ref{eqn:laplace-beltrami}) and
(\ref{eqn:square-norm-gradient}).
\end{lem}

\begin{rem}
 To be precise, we should really write $F_{f,p_0}$ for $F_f$, since the latter also
depends on the reference point. 
\end{rem}

\begin{proof}
As above, let $\rho$ be the metric tensor.
In a local holomorphic chart $(z_a)$ for $P$ centered at $p_0$,
given that $\partial_a\partial_{\overline{b}}f_j=0$ we have
\begin{eqnarray}
 \label{eqn:laplacian-computation}
\Delta_P(F)&=&\rho^{\overline{b}a}\,\partial_a\partial_{\overline{b}}\left(\dfrac{f}{f_1\,f_2}\right)
\nonumber\\
&=&\rho^{\overline{b}a}\,\partial_a\left(\dfrac{1}{f_1\,f_2}\,\partial_{\overline{b}}f
-\dfrac{f}{f_1^2\,f_2}\,\partial_{\overline{b}}f_1\right)\nonumber\\
&=&\rho^{\overline{b}a}\,\left(-\dfrac{1}{f_1\,f_2^2}\,\partial_af_2\,\,\partial_{\overline{b}}f
+\dfrac{1}{f_1\,f_2}\,\partial_a\,\partial_{\overline{b}}f
-\dfrac{1}{f_1^2\,f_2}\,\partial_af\,\partial_{\overline{b}}f_1\right.\nonumber\\
&&\left.+\dfrac{f}{f_1^2\,f_2^2}\,\partial_af_2\,\partial_{\overline{b}}f_1\right).
\end{eqnarray}
At $p_0$, $\partial_af_2(p_0)=\partial_af(p_0)$, 
$\partial_{\overline{b}}f_1(p_0)=\partial_{\overline{b}}f(p_0)$, and
$f_1(p_0)=f_2(p_0)=f(p_0)$. Thus, (\ref{eqn:laplacian-computation}) yields
\begin{eqnarray*}
\Delta_P(F)(p_0)&=&\rho^{\overline{b}a}(p_0)\,
\left(-\dfrac{1}{f(p_0)^3}\,\partial_af(p_0)\,\,\partial_{\overline{b}}f(p_0)
+\dfrac{1}{f(p_0)^2}\,\partial_a\,\partial_{\overline{b}}f(p_0)\right.\nonumber\\
&&\left.
-\dfrac{1}{f(p_0)^3}\,\partial_af(p_0)\,\partial_{\overline{b}}f(p_0)
+\dfrac{1}{f(p_0)^3}\,\partial_af(p_0)\,\partial_{\overline{b}}f(p_0)\right)\nonumber\\
&=&\dfrac{1}{f(p_0)^2}\,\left[\rho^{\overline{b}a}(p_0)\,\partial_a\,\partial_{\overline{b}}f(p_0)
-\dfrac{1}{f(p_0)}\,\rho^{\overline{b}a}(p_0)\,\partial_af(p_0)\,\,\partial_{\overline{b}}f(p_0)\right]
\nonumber\\
&=&\dfrac{1}{f(p_0)^2}\,\left[\Delta_P(f)(p_0)
-\dfrac{1}{2\,f(p_0)}\,\big\|\mathrm{grad}_P(f)(p_0)\big\|_P^2\right].
\end{eqnarray*}

\end{proof}

\begin{lem}
 With the hypothesis and notation of Lemma \ref{lem:technical},
we have
$$
\mathrm{grad}_P(F_f)(p_0)=0.
$$
\label{lem:technical-1}
\end{lem}

\begin{proof}
Let again $(z_a)$ be a local holomorphic
coordinate chart for $P$ centered at $p_0$. Then for every $a$ we have
\begin{eqnarray*}
\lefteqn{\partial_a(F_f)(p_0)=\dfrac{1}{f_1(p_0)}\,\partial_a\left(\dfrac{f}{f_2}\right)(p_0)}\\
&=&\dfrac{f_2(p_0)\,\partial_a f(p_0)-f(p_0)\,\partial_af_2(p_0)}{f_1(p_0)\,f_2(p_0)^2}
=\dfrac{f(p_0)\,\partial_a f(p_0)-f(p_0)\,\partial_af(p_0)}{f(p_0)^3}=0.
\end{eqnarray*}
Similarly, $\partial_{\overline{a}}F_f(p_0)=0$ for every $a$.

\end{proof}

\subsection{The K\"{a}hler structure on $N'$}
\label{sctn:kahler-structure-N'}

We are assuming $\Phi>0$ and $T_{\mathrm{gen}}$ trivial.
Then the two projections $$M'\stackrel{\pi}{\longleftarrow}X'
\stackrel{\kappa}{\longrightarrow}N'$$ are circle bundle structures;
the fibers of $\pi$ are the orbits in $X'$ of $\nu^X$ and those of
$\kappa$ are the orbits in $X'$ of $\mu^X$.

Let $\mathcal{H}=\ker(\alpha)\subseteq TX$ be the horizontal distribution
for $\pi$. Since $\alpha$ is $\mu^X$-invariant, so is $\mathcal{H}$.
In addition, by (\ref{eqn:infinitesimal-lift}) $\mathcal{H}$ is transverse to
every $\mu^X$-orbit. Therefore, it may be viewed as an invariant horizontal
distribution for $\kappa$ as well.

Let $J_{\mathcal{H}}$ be the complex structure that $\mathcal{H}$ inherits
from $J$ by the isomorphism 
$\left.d\pi\right|_{\mathcal{H}}:\mathcal{H}\cong \pi^*(TM)$. Since
$\mu^M$ is holomorphic, $J_{\mathcal{H}}$ is $\mu^X$-invariant. Therefore,
given the isomorphism
$\left.d\kappa\right|_{\mathcal{H}}:\mathcal{H}\cong \kappa^*(TN')$,
it descends to an almost complex structure $I$ on $N'$.

\begin{prop}
 \label{prop:I-is-integrable}
$I$ is a complex structure.
\end{prop}

\begin{proof}
 Let $\mathcal{J}$ be the complex distribution on $M$ associated to $J$.
Thus 
$$
\mathcal{J}=\big\{\mathbf{v}-i\,J(\mathbf{v})\,:\,
\mathbf{v}\in TM\big\}=\ker(J-i\,\mathrm{id})\subseteq TM\otimes \mathbb{C}.
$$ 
As $J$ is integrable, $\mathcal{J}$ is involutive. 

Similarly, let 
$$\mathcal{J}_{\mathcal{H}}=\big\{\mathbf{h}-i\,J_{\mathcal{H}}(\mathbf{h})\,:\,
\mathbf{h}\in \mathcal{H}\big\}=\ker (J_{\mathcal{H}}-i\,\mathrm{id})
\subseteq \mathcal{H}\otimes \mathbb{C}.
$$
Evidently, $\mathcal{J}_{\mathcal{H}}$ is the horizontal lift of $\mathcal{J}_M$.

\begin{lem}
 \label{lem:JH-is-involutive}
$\mathcal{J}_{\mathcal{H}}$ is involutive.
\end{lem}

\begin{proof}
If $V\in \mathfrak{X}(M)$ is real vector field on $M$, then
$U=:V-i\,J(V)$ is a complex vector field on $X$ tangent to
$J_{\mathcal{H}}$, and its horizontal lift
$$
U^\sharp=V^\sharp-i\,J(V)^\sharp=V^\sharp-i\,J_{\mathcal{H}}\left(V^\sharp\right)
$$
is a complex vector field on $X$ tangent to $\mathcal{J} _{\mathcal{H}}$.
It is clear that $\mathcal{J}_{\mathcal{H}}$ is locally spanned by vector fields of this
form, so it suffices to show that $\left[U^\sharp_1,U_2^\sharp\right]$
is tangent to $\mathcal{J}_{\mathcal{H}}$, for any pair of complex vector fields
$U_1$, $U_2$ on $M$ tangent to $\mathcal{J} _{M}$.

Since $\mathcal{J}_M$ is involutive, $[U_1,U_2]$ is tangent to $\mathcal{J}_M$.
Given that $\left[U^\sharp_1,U_2^\sharp\right]$ is $\pi$-correlated to 
$[U_1,U_2]$, to show that $\left[U^\sharp_1,U_2^\sharp\right]$ is 
tangent to $\mathcal{J}_{\mathcal{H}}$ it suffices to show that it is horizontal.

On the one hand, by compatibility of $\omega$ and $J$ and because
by construction $J(U_l)=i\,U_l$, we have
$$
\omega (U_1,U_2)=\omega\big(J(U_1),J(U_2)\big)=i^2\,\omega (U_1,U_2),
$$
so that $\omega (U_1,U_2)=0$. On the other hand, since $U_l^\sharp$ is horizontal
we have $\alpha\left(U_l^\sharp\right)=0$; therefore, given that $d\alpha=2\,\pi^*(\omega)$, we get
\begin{eqnarray*}
0= 2\,\omega(U_1,U_2)&=&2\,\pi^*(\omega)\left(U_1^\sharp,U_2^\sharp\right)
=d\alpha \left(U_1^\sharp,U_2^\sharp\right)\\
&=&U_1^\sharp \cdot \alpha\left(U_2^\sharp\right)-U_2^\sharp \cdot \alpha\left(U_1^\sharp\right)
-\alpha\left(\left[U^\sharp_1,U^\sharp_2\right]\right)\\
&=&-\alpha\left(\left[U^\sharp_1,U^\sharp_2\right]\right).
\end{eqnarray*}

\end{proof}

Finally, let us set
$$
\mathcal{I}=\big\{\mathbf{s}-i\,I(\mathbf{s})\,:\,\mathbf{s}\in TN'\big\}
=\ker (I-i\,\mathrm{id})\subseteq TN'\otimes \mathbb{C}.
$$
We need to prove that $\mathcal{I}$ is an involutive complex distribution. Let
$S_1,\,S_2\in \mathfrak{X}(N')\otimes \mathbb{C}$ be complex vector fields on $N'$
tangent to $\mathcal{I}$, and let $\widehat{S}_1,\,\widehat{S}_2$ be their
horizontal lifts to $X$. By definition of $I$, it follows that the restriction
of $\mathcal{J}_{\mathcal{H}}$ to $X'$ is the horizontal lift of $\mathcal{I}$
under $\kappa$. Therefore, $\widehat{S}_l$ is tangent to $\mathcal{J}_{\mathcal{H}}$
and $\mu^X$-invariant. Then the same holds of their commutator
$\left[\widehat{S}_1,\widehat{S}_2\right]$ because $\mathcal{J}_{\mathcal{H}}$
is involutive and $\mu^X$-invariant. Since $\left[\widehat{S}_1,\widehat{S}_2\right]$
is $\kappa$-correlated to $[S_1,S_2]$, we conclude that $[S_1,S_2]$ is tangent to
$\mathcal{I}$.
\end{proof}

Let us define
\begin{equation}
 \label{eqn:tilde-alpha}
\beta=:\dfrac{1}{\Phi}\,\alpha.
\end{equation}

\begin{lem}
 $\beta$ is a connection form for $\kappa:X'\rightarrow N'$,
with respect to which the horizontal tangent bundle is 
$\mathcal{H}$ (the horizontal tangent bundle of $\pi$).
\end{lem}

\begin{proof}
Since $\mu^X$ preserves $\alpha$ and lifts $\mu^M$,
which is an Hamiltonian action with moment map $\Phi$,
$\beta$ is $\mu^X$-invariant. Furthermore,  
we see from (\ref{eqn:infinitesimal-lift}) and (\ref{eqn:tilde-alpha})
that $\beta(\xi_X)=-1$. 
\end{proof}

Thus $\mathcal{H}\subseteq TX'$ is the horizontal tangent space for
both $\pi$ and $\kappa$. If $V$ is a vector field
on $M$, we shall denote by $V^\sharp$ its horizontal lift to $X$
under $\pi$;
it is a $\nu^X$-invariant section of $\mathcal{H}$ on $X$.
Similarly, if $U$ is a vector field
on $N'$, we shall denote by $\widehat{U}$ its horizontal lift to
$X'$ under $\kappa$; it is a $\mu^X$-invariant section of $\mathcal{H}$
on $X'$. Clearly, vector fields on $M$ are the same as $\nu^X$-invariant
sections of $\mathcal{H}$ on $X$, and vector fields on $N'$ 
are the same as $\mu^X$-invariant
sections of $\mathcal{H}$ on $X$.

\begin{lem}
 \label{lem:kahler-form-N}
There exists a unique K\"{a}hler form $\eta$ on $N'$ such that
$d\beta=2\,\kappa^*(\eta)$.
\end{lem}

\begin{proof}
 We have
\begin{equation}
 \label{eqn:differential-alpha-tilde}
d\beta=\frac{1}{\Phi}\,d\alpha-\frac{1}{\Phi^2}\,d\Phi\wedge \alpha
=\frac{2}{\Phi}\,\pi^*(\omega)-\frac{1}{\Phi^2}\,d\Phi\wedge \alpha,
\end{equation}
and direct inspection using (\ref{eqn:infinitesimal-lift}) shows that
$\iota (\xi_X^\sharp)d\beta=0$.
Since $d\beta$ is $\mu^X$-invariant, it follows that there exists
a necessarily unique 
2-form $\eta$ on $N'$ such that $d\beta=\kappa^*(2\,\eta)$.

Thus, $\eta$ is a closed 2-form on $N'$. To see that it is in fact a K\"{a}hler form,
we need to check that it is compatible with the complex structure and non-degenerate.
To this end, we fix an arbitrary $n\in M'$, choose an arbitrary $x\in \kappa^{-1}(n)$,
and set $m=\pi(x)$. Our construction then yields natural complex-linear isomorphisms
$(T_mM,J_m)\cong (\mathcal{H}_x,J_{\mathcal{H}_x})\cong (T_nN',I_n)$. 
To see that $\eta_n$ is non-degenerate on
$T_nN'$ and compatible with $I_n$, it then suffices to see that the restriction of
$d\beta$ is non-degenerate on $\mathcal{H}_x$, and compatible with
$J_{\mathcal{H}_x}$.

By (\ref{eqn:differential-alpha-tilde}), under the complex-linear isomorphism
$(T_mM,J_m)\cong (\mathcal{H}_x,J_{\mathcal{H}_x})$
the restriction of $d\beta$ on
$\mathcal{H}_x$ may be identified with $2\omega_m/\Phi(m)$ on $T_mM$. Since $\omega$
is K\"{a}hler on $(M,J)$, it is non-degenerate on $T_mM$ and compatible with $J_m$,
and this completes the proof.
\end{proof}

Suppose $f\in \mathcal{C}^\infty(M')^{\mu}\cong \mathcal{C}^\infty(N')^{\nu}$,
and let $H_f$ its Hamiltonian vector field 
on $(M',2\,\omega)$. Since $f$ is $\mu^M$-invariant, so is
$H_f$. Let $H_f^\sharp$  be the horizontal lift
of $H_f$ to $X'$. Then $H_f^\sharp$ is a
$\mu^X\times \nu^X$-invariant horizontal vector field on $X'$, and therefore
it descends to a $\nu^N$-invariant vector field $\overline{H}_f$, 
respectively.

\begin{lem}
 Let $K_f$ be the Hamiltonian vector field of 
$f\in \mathcal{C}^\infty(M')^{\mu} \cong \mathcal{C}^\infty(N')^{\nu}$
on $(N,2\,\eta)$. Then $K_f=\Phi\,\overline{H}_f$.
\end{lem}

\begin{proof}
We need to show that for any $n\in N'$ and $\mathbf{u}\in T_nN'$
we have 
\begin{equation}
 \label{eqn:equality-hamiltonian-vector-field}
2\,\Phi(n)\cdot \eta_n\big(\overline{H}_f(n),\mathbf{u}\big)=d^N_nf(\mathbf{u}),
\end{equation}
where $d^Nf$ is the differential of $f$ when $f$ is viewed as a function
on $N$.

Choose as before $x\in \kappa^{-1}(n)$ and
let $m=:\pi(x)\in M'$. Let $\widehat{\mathbf{u}}\in \mathcal{H}_x$ be the horizontal
lift of $\mathbf{u}$ under $\kappa$, and set
$\mathbf{v}=d_x\pi(\widehat{\mathbf{u}})$. Thus
$\widehat{\mathbf{u}}=\mathbf{v}^\sharp$.
Since $f$ is invariant, up on $X$ we have
$f\circ \pi=f\circ \kappa$; thus, 
\begin{equation}
 \label{eqn:horizontal-vertical-differential}
d^M_mf(\mathbf{v})=d^X_xf\left(\mathbf{v}^\sharp\right)=
d^X_xf\left(\widehat{\mathbf{u}}\right)=
d^N_nf(\mathbf{u}).
\end{equation}

On the other hand, since $H_f^\sharp=\widehat{\overline{H}_f}$, we have
\begin{eqnarray}
 \label{eqn:hamiltonian-part}
\lefteqn{2\,\Phi(n)\cdot \eta_n\left(\overline{H}_f(n),\mathbf{u}\right)}\\
&=&\Phi(m)\cdot d_x\beta\left(H_f(m)^\sharp,\mathbf{v}^\sharp\right)
=\Phi(m)\cdot \frac{1}{\Phi(m)}\,d_x\alpha\left(H_f(m)^\sharp,\mathbf{v}^\sharp\right)\nonumber\\
&=&2\,\omega_m\big(H_f(m),\mathbf{v}\big)=d_m^Mf(\mathbf{v}).\nonumber
\end{eqnarray}
(\ref{eqn:equality-hamiltonian-vector-field}) follows from (\ref{eqn:horizontal-vertical-differential}) and
(\ref{eqn:hamiltonian-part}).

\end{proof}

Suppose $f,g\in \mathcal{C}^\infty(M)^{\mu}$. Since
$\mathcal{C}^\infty(M')^{\mu} \cong \mathcal{C}^\infty(N')^{\nu}$, we
have Poisson brackets $\{f,g\}_M\in \mathcal{C}^\infty(M')^{\mu}$
and $\{f,g\}_N\in \mathcal{C}^\infty(N')^{\nu}$ on $(M',2\,\omega)$
and $(N',2\,\eta)$, respectively. The relation between them under the previous
isomorphism is as follows.

\begin{cor}
 \label{cor:poisson-brackets-MN}
For any $f,\,g\in \mathcal{C}^\infty(M')^{\mu} \cong \mathcal{C}^\infty(N')^{\nu}$,
we have
$\{f,g\}_N=\Phi\,\{f,g\}_M$.
\end{cor}

\begin{proof}
We have, omitting symbols of pull-back,
\begin{eqnarray}
 \label{eqn:poisson-bracket-1st-step}
\{f,g\}_N&=&2\,\eta \big(K_f,K_g\big)=\Phi^2\,d\beta\left(H_f^\sharp.H_g^\sharp\right)\nonumber\\
&=&\Phi^2\,\frac 1\Phi\,d\alpha\left(H_f^\sharp.H_g^\sharp\right)=\Phi\cdot 2\,\omega(H_f,H_g)=
\Phi\cdot \{f,g\}_M.
\end{eqnarray}
\end{proof}

We can similarly relate the gradients $\mathrm{grad}_M(f)$
and $\mathrm{grad}_N(f)$ of an invariant
$f$ on $(M',2\,g)$ and $(N',2\,h)$,
where $g(\cdot,\cdot)=\omega\big(\cdot,J(\cdot)\big)$ and
$h=\eta\big(\cdot,I(\cdot)\big)$ are the Riemannian metrics on $M$ and $N$, respectively.
We have
\begin{equation}
 \label{eqn:gradient-MN}
\mathrm{grad}_N(f)=I\big(K_f\big)=\Phi\,I\left(\overline{H}_f\right)=\Phi\,
\overline{J(H_f)}=\Phi\,\overline{\mathrm{grad}_M(f)}.
\end{equation}

Passing to square norms, we get
\begin{eqnarray}
 \label{eqn:square-norm-gradient-MN}
\big\|\mathrm{grad}_N(f)\|_N^2&=&2\,h\big(\mathrm{grad}_N(f),\mathrm{grad}_N(f)\big)=
2\,\eta\Big(\mathrm{grad}_N(f),I\big(\mathrm{grad}_N(f)\big)\Big)\nonumber\\
&=&2\,\Phi^2\,\eta\left(\overline{\mathrm{grad}_M(f)},\overline{J\big(\mathrm{grad}_M(f)}\big)\right)\nonumber\\
&=&2\,\Phi^2\,d\beta\left(\mathrm{grad}_M(f)^\sharp,J\big(\mathrm{grad}_M(f)^\sharp\right)\nonumber\\
&=&2\,\Phi\,d\alpha\left(\mathrm{grad}_M(f)^\sharp,J\big(\mathrm{grad}_M(f)^\sharp\right)\nonumber\\
&=&\Phi\cdot 2\,\omega\big(\mathrm{grad}_M(f),J\big(\mathrm{grad}_M(f)\big)=\Phi\,\big\|\mathrm{grad}_M(f)\|_M^2.
\end{eqnarray}

\subsection{The descended action on $N$}

Let us dwell on the Hamiltonian nature of the descended action
$\nu^N$. Recall that the action $\nu^X$ given by
(\ref{eqn:action-S-X}), that is, scalar multiplication composed with inversion,
commutes with $\mu^X$, hence it descends to an action
$\nu^N:\mathbf{T}^1\times N\rightarrow N$. 

\begin{lem}
 $\nu^N$ is an holomorphic action
on $(N',I)$.
\label{lem:holomorphic-action-N}
\end{lem}

\begin{proof}
 Choose $n\in N'$ and $x\in \kappa^{-1}(n)$, and let $m=:\pi(x)$. 
Fix $t=e^{i\theta}\in \mathbf{T}^1$.
By construction, we have complex-linear isomorphisms
$T_nN'\cong \mathcal{H}_x\cong T_mM$ that inter-wine
$d_n\nu^N_t:T_nN'\rightarrow T_{\nu^N_t(n)}N'$ with
$d_x\nu^X_t:\mathcal{H}_x\rightarrow \mathcal{H}_{\nu^X_t(x)}=\mathcal{H}_{e^{-i\theta}\cdot x}$,
hence with the identity map of $T_mM$. The statement follows.

\end{proof}

\begin{lem}
 $\nu^N$ is a symplectic action
on $(N',\eta)$.
\label{lem:symplectic-action-N}
\end{lem}

\begin{proof}
 This follows as for Lemma \ref{lem:holomorphic-action-N}, since 
in view of (\ref{eqn:differential-alpha-tilde})
under the previous
isomorphism $\eta_n$ corresponds to $\omega_m/\Phi(m)$.

\end{proof}

Thus $\nu^N_t$ is an automorphism of the K\"{a}hler manifold $(N',I,\eta)$, for each $t\in \mathbf{T}^1$.

\begin{lem}
\label{lem:hamiltonian-action-N-inverse}
 $\nu^N$ is an Hamiltonian action on $(N',2\,\eta)$, with moment map $1/\Phi$ (viewed as a function on $N$).
\end{lem}

\begin{proof}
The vector field $-\partial/\partial \theta$ on $X$ is $\mu^X$-invariant, hence it
descends to a vector field $\upsilon$ on $N$, which is the infinitesimal generator
of $\nu^X$. We need to show that $2\,\iota(\upsilon)\,\eta=d^N(1/\Phi)$, that is,
for any $n\in N'$ and any $\mathbf{u}\in T_nN'$
we have
\begin{equation}
 \label{eqn:hamiltonian-descended}
2\,\eta_n(\upsilon,\mathbf{u})=-\Phi(n)^{-2}\,d^N_n\Phi(\mathbf{u}).
\end{equation}
 
As before, let $\widehat{\mathbf{u}}$ be the horizontal lift of $\mathbf{u}$ with respect to $\kappa$,
and set $\mathbf{v}=:d_x\pi\left(\widehat{\mathbf{u}}\right)$,
so that $\widehat{\mathbf{u}}=\mathbf{v}^\sharp$.
Thus
\begin{eqnarray}
 \label{eqn:hamiltonian-descended-0}
-\Phi(n)^{-2}\,d^N_n\Phi(\mathbf{u})&=&
-\Phi(x)^{-2}\,d^X_x\Phi\left(\widehat{\mathbf{u}}\right)\\
&=&-\Phi(x)^{-2}\,d^X_x\Phi\left(\mathbf{v}^\sharp\right)=-\Phi(m)^{-2}\,d^M_m\Phi\left(\mathbf{v}\right).\nonumber
\end{eqnarray}

On the other hand, since
$\kappa^*(2\,\eta)=d\beta$, we have
\begin{eqnarray}
 \label{eqn:hamiltonian-descended-1}
d_x\beta\left(-\dfrac{\partial}{\partial \theta},\mathbf{v}^\sharp\right)=
d_x\beta\left(-\dfrac{\partial}{\partial \theta},\widehat{\mathbf{u}}\right)=2\,\eta_n(\upsilon,\mathbf{u}).
\end{eqnarray}

Then (\ref{eqn:hamiltonian-descended}) is equivalent to the equality
\begin{eqnarray}
 \label{eqn:hamiltonian-descended-2}
d_x\beta\left(\dfrac{\partial}{\partial \theta},\mathbf{v}^\sharp\right)=
\dfrac{1}{\Phi(m)^2}\,d^M_m\Phi\left(\mathbf{v}\right),
\end{eqnarray}
for any $m\in M$, $\mathbf{v}\in T_mM$, and $x\in \pi^{-1}(x)$. The latter is an immediate consequence 
of (\ref{eqn:differential-alpha-tilde}).
\end{proof}

Now $\beta$ is a connection 1-form for the circle bundle
$\kappa:X'\rightarrow N'$ and is preserved by $\nu^X$; therefore, for an appropriate
constant $c$, $\nu^X$ is a contact lift
to $\left(X',\beta\right)$ of $\nu^N$, with respect to
the Hamiltonian $c+1/\Phi$.

\begin{lem}
 \label{lem:c=0}
The correct choice is $c=0$. Furthermore, the horizontal lift
of $\upsilon$ with respect to $\kappa$ is
$$
\widehat{\upsilon}=-\frac{1}{\Phi}\,\xi_M^\sharp.
$$
\end{lem}

\begin{proof}
We want to give a decomposition of $-\partial/\partial \theta$
analogous to (\ref{eqn:infinitesimal-lift}), but referred to 
the circle bundle structure $\kappa:X'\rightarrow N'$.
To this end, let $\beta$ be a locally defined angular coordinate on $X'$
referred to $\kappa$, so that $\xi_X=-\partial/\partial \beta$.
Since the horizontal component of $-\partial/\partial \theta$
with respect to $\kappa$ is $\widehat{\upsilon}$, the analogue
of (\ref{eqn:infinitesimal-lift}) is
\begin{eqnarray}
 \label{eqn:infinitesimal-lift-prime}
-\dfrac{\partial}{\partial \theta}&=&
\widehat{\upsilon}-\left(c+\frac{1}{\Phi}\right)\,\dfrac{\partial}{\partial \beta}
\nonumber\\
&=&\widehat{\upsilon}+\left(c+\frac{1}{\Phi}\right)\,
\left(\xi_M^\sharp-\Phi\,\dfrac{\partial}{\partial \theta}\right)\nonumber\\
&=&\left[\widehat{\upsilon}+\left(c+\frac{1}{\Phi}\right)\,\xi_M^\sharp\right]
-(1+c\,\Phi)\,\dfrac{\partial}{\partial \theta},
\end{eqnarray}
where the latter is a decomposition into horizontal and vertical
components with respect to $\pi$.
The latter equality is equivalent to the claimed statement.

\end{proof}

\subsection{The complexified action on $A^\vee_0$}

The action $\mu^M:\mathbf{T}^1\times M\rightarrow M$ complexifies
to an holomorphic action $\widetilde{\mu}^M:\mathbb{T}^1\times M\rightarrow M$,
where $\mathbb{T}^1=\mathrm{GL}(1,\mathbb{C})\cong \mathbb{C}^*$ (see, for instance,
the discussion in \S4 of \cite{guillemin-sternberg}).
Let $(\rho,\vartheta)$ be polar coordinates on $\mathbf{C}^*$,
and let $\xi=:\partial/\partial \vartheta$, $\eta=:-\rho\,\partial/\partial\rho$;
then $\eta=J_0(\xi)$ ($J_0$ being the complex structure on 
$\mathbf{C}^*$). By holomorphicity, if $\xi_M$ and $\eta_M$ are the induced vector
fields on $M$, then $\eta_M=J_M(\xi_M)$

On the other hand, the contact lift $\mu^X:\mathbf{T}^1\times X\rightarrow X$
of $\mu^M$ extends to a linearized action $\mu^{A^\vee}:\mathbf{T}^1\times A^\vee_0\rightarrow
A^\vee_0$. There is a natural diffeomorphism $X\times \mathbb{R}_+\cong A^\vee_0$,
given by $(x,r)\mapsto r\cdot x$; as a function on $A^\vee_0$, $r$ is simply
the norm for the given Hermitian structure. If $\theta$ is a locally defined angular
coordinate on $X$, depending on the choice of a local unitary frame of
$A^\vee$, then $(r,\theta)$ restrict to polar coordinates along the fibers of
$A^\vee_0$. Thus, if $J_{A^\vee}$ is the complex structure of $A^\vee$, then
the globally defined vertical vector fields $\partial/\partial \theta$ and $\partial/\partial r$
on $A^\vee_0$ are related by $J_{A^\vee}(\partial/\partial \theta)=-r\,\partial/\partial r$.
By (\ref{eqn:infinitesimal-lift}), the infinitesimal generator of $\mu^{A^\vee}$
is 
\begin{equation}
 \label{eqn:infinitesimal-generator-line-bundle}
\xi_{A^\vee}=\xi_M^\sharp-\Phi\,\dfrac{\partial}{\partial \theta},
\end{equation}
where the horizontal lift is now taken in the tangent bundle of $A^\vee$, with respect to
the extended connection.

The action $\mu^{A^\vee}:\mathbf{T}^1\times A^\vee_0\rightarrow
A^\vee_0$ again extends to an holomorphic action 
$\widetilde{\mu}^{A^\vee}:\mathbb{T}^1\times A^\vee_0\rightarrow
A^\vee_0$, which is of course a linearization of $\widetilde{\mu}^M$ (see the discussion in
\S 5 of \cite{guillemin-sternberg}).
By holomorphicity, the induced vector fields $\xi_{A^\vee}$ and $\eta_{A^\vee}$, with 
$\xi_{A^\vee}$ given by (\ref{eqn:infinitesimal-generator-line-bundle}), satisfy
\begin{equation}
 \label{eqn:second-infinitesimal-generator-line-bundle}
\eta_{A^\vee}=J_{A^\vee}(\xi_{A^\vee})=\eta_M^\sharp+\Phi\,r\,\dfrac{\partial}{\partial r}.
\end{equation}

Let $\mathcal{N}_A:A^\vee_0\rightarrow \mathbb{R}$ be the square norm function;
thus $\mathcal{N}_A=r^2$ under the previous diffeomorphism
$A^\vee_0\cong X\times \mathbb{R}_+$. Then 
\begin{equation}
 \label{eqn:square-norm-derivation}
\xi_{A^\vee}(\mathcal{N}_A)=0,\,\,\,\,\,\,\,\,\,\eta_{A^\vee}(\mathcal{N}_A)=2\,\Phi\,\mathcal{N}_A>0.
\end{equation}

\begin{lem}
 \label{lem:estimate-norm-moment}
Let $a=:\min |\Phi|$, $A=:\max|\Phi|$. Then, for every $\lambda\in A^\vee_0$,
we have
$$
e^{2a\,t}\,\mathcal{N}_A(\lambda)\le 
\mathcal{N}_A\left(\widetilde{\mu}^{A^\vee}_{e^{-t}}\left(\lambda\right)\right)
\le e^{2A\,t}\,\mathcal{N}_A(\lambda)
$$
if $t\ge 0$, and
$$
e^{2A\,t}\,\mathcal{N}_A(\lambda)\le \mathcal{N}_A\left(\widetilde{\mu}^{A^\vee}_{e^{-t}}\left(\lambda\right)\right)
\le e^{2a\,t}\,\mathcal{N}_A(\lambda)
$$
if $t<0$.
\end{lem}

\begin{proof}
 The invariant vector field $\eta=-\rho\,\partial/\partial\rho$ on $\mathbb{C}^*$ is associated
to the 1-parameter subgroup $t\mapsto e^{-t}$. Therefore, if given $\lambda\in A^\vee_0$
we define 
$$
\mathcal{N}_A^\lambda(t)=:\mathcal{N}_A\left(\widetilde{\mu}^{A^\vee}\left(e^{-t},\lambda\right)\right)
\,\,\,\,\,\,\,\,\,\,\,\,\,\,\,(t\in \mathbb{R})
$$
then by (\ref{eqn:square-norm-derivation})
\begin{eqnarray*}
\frac{d}{dt}\mathcal{N}_A^\lambda(t)=
\eta_{A^\vee}\big(\mathcal{N}_A\big)\left(\widetilde{\mu}^{A^\vee}\left(e^{-t},\lambda\right)\right)
=2\,\Phi\left(\widetilde{\mu}^{A^\vee}_{e^{-t}}\left(\lambda\right)\right)\,\mathcal{N}_A^\lambda(t),
\end{eqnarray*}
which can be rewritten 
\begin{equation}
 \label{eqn:logarithmic-derivation}
\left.\dfrac{d}{dt}\ln (\mathcal{N}_A^\lambda)\right|_{t=t_0}=2\,
\Phi\left(\widetilde{\mu}^{A^\vee}_{e^{-t_0}}\left(\lambda\right)\right)
\end{equation}
for any $t_0\in \mathbb{R}$.
We deduce from (\ref{eqn:logarithmic-derivation}) that
\begin{equation}
 \label{eqn:derivation-estimate-logarithmic}
2\,a\le \dfrac{d}{dt}\ln (\mathcal{N}_A^\lambda)\le 2\,A,
\end{equation}
which easily implies the claim.

\end{proof}

Let us set, for $(z,\lambda)\in \mathbb{C}^*\times A^\vee_0$:
\begin{equation}
 \label{eqn:bullet-action}
z\bullet \lambda=:\widetilde{\mu}^{A^\vee}\left(z^{-1},\lambda\right)
=\widetilde{\mu}^{A^\vee}_{z^{-1}}\left(\lambda\right).
\end{equation}

\begin{cor}
\label{cor:norm-product-z}
 If $|z|\ge 1$, then
$$
|z|^{2a}\,\mathcal{N}_A(\lambda)\le \mathcal{N}_A\left(z\bullet\lambda\right)
\le |z|^{2A}\,\mathcal{N}_A(\lambda).
$$
If $0<|z|<1$, then
$$
|z|^{2A}\,\mathcal{N}_A(\lambda)\le \mathcal{N}_A\left(z\bullet\lambda\right)
\le |z|^{2a}\,\mathcal{N}_A(\lambda).
$$
\end{cor}

\begin{proof}
 If $z=e^{t+is}$, with $t,\,s\in \mathbb{R}$ and $|z|=e^t$, 
then because the action 
of $e^{is}\in \mathbf{T}^1$ is metric preserving we have
$$
\mathcal{N}_A\left(z\bullet\lambda\right)=\mathcal{N}_A\left(e^t\bullet\lambda\right)
=\mathcal{N}_A\left(\widetilde{\mu}^{A^\vee}_{e^{-t}}\left(\lambda\right)\right).
$$
Thus the Corollary is just a restatement of Lemma \ref{lem:estimate-norm-moment}.
\end{proof}

\begin{cor}
 \label{cor:surjective-product-map}
The $\mathcal{C}^\infty$ map $\varUpsilon:\mathbb{R}_+\times X'\rightarrow {A^\vee_0}'$
given by $(t,x)\mapsto t\bullet x$ is a bijection.
\end{cor}

\begin{proof}
By Corollary \ref{cor:norm-product-z}, $\mathcal{N}_A(t\bullet \lambda)\rightarrow +\infty$
as $t\rightarrow +\infty$, and $\mathcal{N}_A(t\bullet \lambda)\rightarrow 0^+$
as $t\rightarrow 0^+$; thus for any $\lambda\in {A^\vee_0}'$ there exists $t_\lambda\in \mathbb{R}_+$
such that $t_\lambda^{-1}\bullet\lambda\in X'$. Corollary \ref{cor:norm-product-z} also implies that
$\mathcal{N}_A\left(t\bullet x\right)>\mathcal{N}_A\left(x\right)$ for any
$t>1$ and $x\in X'$. Therefore,
$t\mapsto \mathcal{N}_A\left(t\bullet\lambda\right)$ is a strictly increasing function, since
if $t_1<t_2$ then 
\begin{eqnarray}
\mathcal{N}_A\left(t_2\bullet \lambda\right)&=&
\mathcal{N}_A\left(\left(\frac{t_2}{t_1}\,t_1\right)\bullet \lambda\right)\nonumber\\
&=&\mathcal{N}_A\left(\left(\frac{t_2}{t_1}\right)\bullet (t_1\bullet \lambda)\right)>
\mathcal{N}_A\left(t_1\bullet \lambda\right).
\end{eqnarray}
Hence
$t_\lambda$ is in fact unique.
\end{proof}

\begin{rem}
\label{rem:real-analytic-equivalence}
More is true. 
Since $X'\subseteq A^\vee_0$ is a real-analytic submanifold (see Corollary
\ref{cor:analytic-N-1}), $\mathbb{R}_+\times X'$ is a real-analytic submanifold
of $\mathbb{C}^*\times A^\vee_0$. Being the restriction of the holomorphic map
(\ref{eqn:bullet-action}), $\varUpsilon$ is then a real-analytic bijection of real-analytic
manifolds. It is in fact also a local diffeomorphism, for its differential has everywhere
maximal rank; by the real-analytic inverse function theorem (Theorem 2.5.1 of
\cite{parks-krantz}), $\varUpsilon$ is a 
real-analytic equivalence between $\mathbb{R}_+\times X'$
and ${A^\vee_0}'$
\end{rem}

Recall that a Lie group action on a manifold $P$ is called \textit{proper}
if the associated \textit{action map} $G\times P\rightarrow P\times P$, $(g,p)\mapsto (g\cdot p,p)$,
is proper (Definition B2 of \cite{guillemin-ginzburg-karshon}).

Let ${A^\vee_0}'\subseteq A^\vee_0$ be the inverse image of $M'$; in other words,
in terms of the diffeomorphism $A^\vee_0\cong X\times \mathbb{R}_+$, we have
${A^\vee_0}'\cong X'\times \mathbb{R}_+$.

\begin{cor}
\label{cor:proper-action}
The complexified action $\widetilde{\mu}^{A^\vee}:\mathbb{T}^1\times A^\vee_0\rightarrow A^\vee_0$
is proper. In addition, its restriction to ${A^\vee_0}'$ is free.
\end{cor}

\begin{proof}
 Let $\Upsilon:\mathbb{T}^1\times A^\vee_0\rightarrow A^\vee_0\times A^\vee_0$ be the action map of
$\widetilde{\mu}^{A^\vee}$, and let $R\subset A^\vee_0\times A^\vee_0$ be a compact subset.
If $\pi_j:A^\vee_0\times A^\vee_0\rightarrow A^\vee_0$ is the projection onto the $j$-th factor,
let $R_j=:\pi_j(R)$. Then $R_j$ is compact  and $R\subset R_1\times R_2$.
Therefore, to prove that $\Upsilon$ is proper, it suffices to show that 
$\Upsilon^{-1}(R_1\times R_2)$
is compact, for any pair of compact subsets $R_1, \,R_2\subset A^\vee_0$.
Clearly, $\Upsilon^{-1}(R_1\times R_2)\subseteq  \mathbb{T}^1\times R_2$. For $j=1,2$
let 
$\ell_j=:\min_{R_j}\mathcal{N}_A$ and $L_j=:\max_{R_j}\mathcal{N}_A$. 

Suppose $(w,\lambda)\in \Upsilon^{-1}(R_1\times R_2)$,
and set $z=w^{-1}$, so that $\widetilde{\mu}^{A^\vee}(w,\lambda)=z\bullet \lambda$.

If $|w|\le 1$, then $|z|\ge 1$ and so by Corollary \ref{cor:norm-product-z}
$$
|z|^{2a}\,\mathcal{N}_A(\lambda)\le \mathcal{N}_A(z\bullet \lambda)
\le |z|^{2A}\,\mathcal{N}_A(\lambda).
$$
Since $\lambda\in R_2$, we have $l_2\le\mathcal{N}_A(\lambda)$, and since
$z\bullet \lambda\in R_1$, we have $\mathcal{N}_A(z\bullet \lambda)\le L_1$.
Therefore, if $|z|\ge 1$ then
$$
\ell_2\,|z|^{2a}\le L_1\,\,\,\,\,\Longrightarrow\,\,\,\,|z|\le 
\left(\frac{L_1}{l_2}\right)^{1/2a}.
$$
In other words, if $|w|\le 1$ then $(\ell_2/L_1)^{1/2a}\le |w|$.
Similarly, one sees that if $|w|\ge 1$ then
$|w|\le (L_2/\ell_1)^{1/2a}$.

Therefore, if $p:\mathbb{T}^1\times A^\vee_0 \rightarrow A^\vee_0$
is the projection onto the first factor, then
$S\left(\Upsilon^{-1}(R_1\times R_2)\right)$ is compact.
Therefore, $\Upsilon^{-1}(R_1\times R_2)\subseteq S\times R_2$ is also
compact, and this completes the proof that the action is proper.

The statement about the freeness of the action follows immediately from
(\ref{eqn:derivation-estimate-logarithmic}) and the definition of $X'$.

\end{proof}

\subsection{The circle bundle structures}
\label{sctn:circle-bundle-structure}

Let us view $\kappa:X'\rightarrow N'$ as a circle bundle over $N'$,
with the action of $\mathbf{T}^1$ on $X'$ given by 
$\left(e^{i\theta},x\right)\mapsto e^{i\theta}\bullet x$;
the latter is defined in (\ref{eqn:bullet-action}).
On $N'$, associated to the K\"{a}hler
structure we have the volume form $dV_N=\eta^{\wedge d}/d!$; thus
on $X'$, viewed as a circle bundle over $N'$, we have the natural choice
of a volume form 
$dW_X=(1/2\pi)\,\beta\wedge \kappa^*(dV_N)$.
Algebraically, $L^2(X,dV_X)= L^2(X',dV_X)=L^2(X',dW_X)$, although the
metrics are different. Explicitly,
\begin{eqnarray}
 \label{eqn:comparison-volume-forms}
dW_X&=&\dfrac{1}{2\pi}\,\,\beta\wedge \kappa^*(dV_N)\nonumber\\
&=&\dfrac{1}{2\pi}\,\dfrac{\alpha}{\Phi}\wedge \frac{1}{d!}\,
\left(\frac{1}{\Phi}\,\pi^*(\omega)-\frac{1}{2\Phi^2}\,d\Phi\wedge \alpha\right)^{\wedge d}\nonumber\\
&=&\Phi^{-(d+1)}\,\left[\dfrac{1}{2\pi}\,\alpha\wedge \pi^*(\omega)^{\wedge d}\right]\nonumber\\
&=&\Phi^{-(d+1)}\,dV_X.
\end{eqnarray}

Furthermore, the two circle bundles $\pi$ and $\kappa$
have
different CR structures, because they do not have the same vertical tangent
bundle. However, by construction
they share the same horizontal distribution, and the
same horizontal complex structure $J_{\mathcal{H}}$.
Let $\mathcal{H}^{(0,1)}\subseteq \mathcal{H}\otimes \mathbb{C}$ be the $-i$-eigenbundle of 
$J_{\mathcal{H}}$; then the 
the boundary CR operator of either $X$ or $X'$ is defined by
setting $\overline{\partial}_bf=\left.df\right|_{\mathcal{H}^{(0,1)}}$,
for any $\mathcal{C}^\infty$ function $f$ on $X$ or $X'$,
respectively. Therefore, the boundary CR operator of $X$,
$$\overline{\partial}_b:\mathcal{C}^\infty(X)\rightarrow 
\mathcal{C}^\infty\left(X,{\mathcal{H}^{(0,1)}}^\vee\right),$$
restricts to the corresponding operator of $X'$. It follows that there is a natural
algebraic (non isometric) inclusion of corresponding Hardy spaces, $H(X)\hookrightarrow H(X')$.
The latter is an algebraic isomorphism if
$\mathrm{codim}_{\mathbb{C}}\big(M\setminus M',M\big)\ge 2$.

The action $\mu^X$ plays the role of the structure circle action of $\mathbf{T}^1$
with respect to $\kappa$. Let $\widetilde{H}_k(X')$ be the $k$-th
isotype for the latter action.
Condition (\ref{eqn:isotypes}) for $s\in L^2(X)$
to belong to $H^\mu_k(X)$ may be rewritten 
$s\left(e^{i\theta}\bullet x\right)=e^{ik\theta}\,s(x)$, for any 
$e^{i\theta}\in \mathbf{T}^1$ and $x\in X$. Therefore, the previous inclusion of
Hardy spaces yields for every $k=0,1,2,\ldots$ an
algebraic inclusion $H^\mu_k(X)\hookrightarrow \widetilde{H}_k(X')$.

\subsection{The line bundle on $N'$}

Let $B$ be the complex line bundle on $N'$ associated to 
$\kappa$ and the tautological action of
$\mathbf{T}^1=U(1)$ on $\mathbb{C}$, and let $B^\vee$ be its dual; 
thus, $B$ (respectively, $B^\vee$)
is the quotient of
$X'\times \mathbb{C}$ by the equivalence relation
$(x,w)\sim \left(e^{i\theta}\bullet x,e^{i\theta}\,w\right)$
(respectively, $(x,w)\sim \left(e^{i\theta}\bullet x,e^{-i\theta}\,w\right)$).
We can embed $\jmath:X'\hookrightarrow B^\vee$ by $x\mapsto [x,1]$.
Then $B$ and $B^\vee$ inherit natural Hermitian structures, that we shall
denote by $\ell_B$,
uniquely determined by imposing $\jmath:X'\hookrightarrow B^\vee$ to embed as the unit
circle bundle. We shall denote by $\widehat{\kappa}:B^\vee\rightarrow N'$
the projection, so that $\kappa=\widehat{\kappa}\circ \jmath$.

The connection form $\beta$ on $X$ determines a unique
metric covariant derivative $\nabla_B$ on $B$, with curvature
$\Theta_B=-2i\,\kappa^*(\eta)$. Since $\eta$ is a K\"{a}hler form on $N'$,
there is a uniquely determined holomorphic structure on $B$, such that
$\nabla_B$ is the only covariant derivative on $B$ compatible with both
the metric and the latter holomorphic structure. A local section $\sigma$ of
$B$ is holomorphic for this structure if and only if the connection matrix
with respect to $\sigma$ is of type $(1,0)$.

\begin{lem}
 \label{lem:natural-isomorphism_AB}
There is a natural biholomorphism $\overline{\Gamma}:B^\vee_0\cong {A^\vee_0}'$
of bundles over $N'$; when we view 
$X'$ as a submanifold of ${A^\vee_0}'$ and $B^\vee_0$ in the natural manners, $\overline{\Gamma}$
restricts to the identity $X'\rightarrow X'$ (that is, $\overline{\Gamma}\big(\jmath(x)\big)=x$
for any $x\in X'$). Furthermore, $\overline{\Gamma}$ preserves the horizontal distributions,
and maps biholomorphically the fibers of the bundle projection
$\widehat{\kappa}:B^\vee\rightarrow N'$ onto the orbits of the action
$\widetilde{\mu}^{A^\vee}:\mathbb{T}^1\times A^\vee_0\rightarrow A^\vee_0$.
\end{lem}

\begin{proof}
We have $B_0^\vee=X'\times \mathbb{C}^*/\sim$,
where $(x,w)\sim \left(e^{i\theta}\bullet x,e^{-i\theta}\,w\right)$,
for any $e^{i\theta}\in \mathbf{T}^1$. 
If $\varPsi:X'\times \mathbb{C}^*\rightarrow B_0^\vee$
is the quotient map, $\varPsi(x,w)=[x,w]$,
for any $(x,w)\in X'\times \mathbb{C}^*$
the differential $d_{(x,w)}\varPsi$ induces a $\mathbb{C}$-linear isomorphism
$$
\mathcal{H}_x\oplus \mathbb{C}\subset T_{(x,w)}\left(X'\times \mathbb{C}^*\right)
\cong T_{\varPsi(x,w)}B_0^\vee,
$$
which maps $\mathcal{H}_x\oplus (\mathbf{0})$ and $(\mathbf{0})\oplus \mathbb{C}$,
respectively,
onto the horizontal and vertical tangent spaces of $B^\vee$ at $\varPsi(x,w)$.

Let us consider the map 
$$
\Gamma:X'\times \mathbb{C}^*\rightarrow {A^\vee_0}',\,\,\,\,\,\,\,\,\,
(x,w)\mapsto w\bullet x.
$$
Holomorphicity of the complexified action $\widetilde{\mu}$ implies that the
differential $d_{(x,w)}\Gamma$ induces a $\mathbb{C}$-linear isomorphism
\begin{equation}
 \label{eqn:isomorphism_AB-down}
\mathcal{H}_x\oplus \mathbb{C}\subset T_{(x,w)}\left(X'\times \mathbb{C}^*\right)
\cong T_{\Gamma(x,w)}A_0^\vee,
\end{equation}
under which $\mathcal{H}_x\oplus \mathbb{C}$ maps onto the horizontal tangent space
of $A^\vee$
and $(\mathbf{0})\oplus \mathbb{C}$ onto the tangent space to the complex orbit
$\widetilde{\mu}^{A^\vee}$ at $\Gamma(x,w)$.

On the other hand, 
for any $(x,w)\in X'\times \mathbb{C}^*$ and $e^{i\theta}\in \mathbf{T}^1$,
we have
\begin{eqnarray}
 \label{eqn:associative-action-quotient}
\Gamma(x,w)=w\bullet x=\left(w\,e^{-i\theta}\right)\bullet \left(e^{i\theta}\,x\right)
=\Gamma\left(e^{i\theta}\,x,e^{-i\theta}\,w\right).
\end{eqnarray}
Therefore, $\Gamma$ passes to the quotient under $\varPsi$, that is, there exists a $\mathcal{C}^\infty$
map $\overline{\Gamma}:B_0^\vee\rightarrow {A^\vee_0}'$ such that
$\Gamma=\overline{\Gamma}\circ \varPsi$, that is,
$$
\overline{\Gamma}([x,w])=\Gamma(x,w)\,\,\,\,\,\,\,\,\,\,\,\,\left(
(x,w)\in X'\times \mathbb{C}^*\right);
$$
the previous discussion implies that
$\overline{\Gamma}$ is holomorphic. 

Corollary \ref{cor:surjective-product-map} evidently implies that 
$\Gamma$ is surjective, and therefore so is $\overline{\Gamma}$.
To see that $\overline{\Gamma}$ is also injective, suppose that
$\lambda_j=\varPsi(x_j,w_j)$, $j=1,2$, 
satisfy $\overline{\Gamma}(\lambda_1)=\overline{\Gamma}(\lambda_2)$.
Thus $w_1\bullet x_1=\Gamma(x_1,w_1)=\Gamma_2(x_2,w_2)=w_2\bullet x_2$, 
whence $\left(w_2^{-1}\,w_1\right)\bullet x_1=x_2$. This evidently implies
$\mathcal{N}_A\left(\left(w_2^{-1}\,w_1\right)\bullet x_1\right)=\mathcal{N}_A(x_1)=
\mathcal{N}_A(x_2)=1$. Since by Corollary \ref{cor:norm-product-z}
the map $t\mapsto \mathcal{N}_A(t\bullet x_1)$ is strictly increasing, this
forces $|w_1|=|w_2|$. If $w_2^{-1}\,w_1=e^{i\theta}$, we then have
$w_2=e^{-i\theta}\,w_1$, $x_2=e^{i\theta}\bullet x_1$; hence
$(x_2,w_2)=\left(e^{i\theta}\bullet x_1,e^{-i\theta}\bullet w_1\right)
\sim (x_1,w_1)$. Therefore, $\lambda_1=\lambda_2$.

Finally, any $x\in X'\subset {A^\vee_0}'$ corresponds to $[x,1]\in B^\vee_0$,
and $\overline{\Gamma}([x,1])=1\bullet x=x$. Therefore, with the previous
identification $\overline{\Gamma}$ induces the identity map on $X'$.
\end{proof}

\begin{rem}
\label{rem:natural-isomorphism_AB}
$\overline{\varGamma}$
interwines fiberwise scalar multiplication $\cdot_B$
on $B^\vee_0$ and the map 
(\ref{eqn:bullet-action}). In fact, if $b=[x,w]\in B^\vee_0$ and $z\in \mathbb{C}^*$,
then $z\cdot _Bb=[x,z\,w]$. Therefore, 
$$
\overline{\varGamma}(z\cdot_B b)=
\overline{\varGamma}\big([x,z\,w]\big)=(z\,w)\bullet x=z\bullet (w\bullet x)
=z\bullet \overline{\varGamma}(x).
$$
\end{rem}

Let $\mathcal{N}_B:A^\vee_0\rightarrow \mathbb{R}$ be the norm function associated
to the Hermitian structure of $B^\vee$, viewed as a function on $A^\vee_0$ by 
means of the biholomorphism of Lemma \ref{lem:natural-isomorphism_AB}. Then
$\mathcal{N}_B(z\bullet \lambda)=|z|^2\,\|\lambda\|^2$.

\begin{cor}
 \label{cor:analytic-N-1}
$X'\subseteq A^\vee_0$ is a real-analytic submanifold, and the projection
$\kappa:X'\rightarrow N'$ is real-analytic.
\end{cor}

\begin{proof}
 Since the Hermitian metric $h$ on $A^\vee$ is real-analytic by assumption,
the norm function $\mathcal{N}_A:{A^\vee_0}\rightarrow \mathbb{R}$ is a positive
real-analytic function. Therefore, $X'=\mathcal{N}_A^{-1}(1)\cap {A^\vee_0}'$
is a real-analytic submanifold of ${A^\vee_0}'$ (see \S 2.7 of \cite{parks-krantz}).
On the other hand, we have
$$
\kappa=\widehat{\kappa}\circ \jmath=
\left.\widehat{\kappa}\circ \overline{\Gamma}^{-1}\right|_X.
$$
Thus $\kappa$ is the restriction of an holomorphic map to a real-analytic submanifold,
hence it is real-analytic.
\end{proof}

\begin{prop}
 \label{prop:real-analytic-eta}
The K\"{a}hler form $\eta$ on $N'$ is real-analytic.
\end{prop}

\begin{proof}
 It suffices to prove that for any $n\in N'$ there is a real-analytic
chart for $N'$, defined on an open neighborhood $V\subseteq N'$ of $n$,
such that the local expression of $\eta$ in that chart is real-analytic.
To this end, choose $x\in \kappa^{-1}(n)$ and let $m=:\pi(x)\in M'$.
On some open neighborhood $U\subseteq M$ of $m$, we can find a local holomorphic
frame $\varphi$ on $A^\vee$ such that $\varphi(m)=x$. 
We can also suppose, without loss of generality, that
$\varphi$ is horizontal at $m$, that it, its differential at $m$ maps
isomorphically $T_mM$ to the horizontal tangent space $\mathcal{H}_x$.
If we assume, as we may, that $U$ is the domain of a holomorphic local coordinate
chart $(z_j)$ centered at $m$, 
this means that $\mathcal{N}_A\circ \varphi=h(\varphi,\varphi)=1+O\left(\|z\|^2\right)$.

Let us write $\|\varphi\|=:\sqrt{h(\varphi,\varphi)}$, so that 
$\zeta=:\varphi/\|\varphi\|:U\rightarrow X'
$
is a local unitary frame. Since $h$ is real-analytic, $\zeta$ is real-analytic. 
The previous remark shows, in addition, that
$$
d_m\zeta=d_m\varphi:T_mM\longrightarrow T_xX\subseteq T_xA^\vee;
$$ 
therefore $\zeta$ is also horizontal at $m$,
whence it is transverse at $m$ 
to the $\mu^X$-orbit through $x$, in view of (\ref{eqn:infinitesimal-lift})
and the positivity of $\Phi$. Since the latter orbit is the fiber 
through $x$ of the projection 
$\kappa:X'\rightarrow N'$,
this implies that the composition $\kappa\circ \zeta:U\rightarrow N'$ is a real-analytic
local diffeomorphism at $m$; we have
$\kappa\circ \zeta(m)=\kappa (x)=n$. Therefore, perhaps after replacing $U$ with a smaller open
neighborhood of $m$, we may assume that $\kappa\circ \zeta$ induces a real-analytic
equivalence $U\cong V$, where $V=: \kappa\circ \zeta(U)$ is an open neighborhood
of $n$ (see Theorem 2.5.1 of \cite{parks-krantz}). 
Given the holomorphic chart on $U$, we may then interpret
$\kappa\circ \zeta$ as a real-analytic chart for $N'$ in the neighborhood of $n$.

Let $\theta_\zeta^\vee=i\,\zeta^*(\alpha)$ be the connection form of $A^\vee$ in the local frame
$\zeta$. Then under our assumptions $\theta_\zeta^\vee$ is a real-analytic imaginary 1-form.
The local expression of $2\,\eta$ in this chart, by (\ref{eqn:differential-alpha-tilde}),
is
\begin{eqnarray}
 \label{eqn:local-expression-eta}
(\kappa\circ \zeta)^*(2\eta)&=&\zeta^*\left(\kappa^*(2\,\eta)\right)\nonumber\\
&=&\zeta^*\left(\frac{2}{\Phi}\,\pi^*(\omega)-\frac{1}{\Phi^2}\,d\Phi\wedge \alpha\right)\nonumber\\
&=&\frac{2}{\Phi}\,\omega+\frac{i}{\Phi^2}\,d\Phi\wedge \theta_\zeta^\vee.
\end{eqnarray}
In view of Lemma \ref{lem:real-analytic-moment-map},
we conclude that (\ref{eqn:local-expression-eta})
is real-analytic, and this completes the proof.
\end{proof}

This also follows from the following:

\begin{lem}
 \label{lem:real-analytic-herm-metric-B}
The Hermitian metric $h$ on $B^\vee$ is real-analytic.
\end{lem}

\begin{proof}
 It suffices to show that the norm function $\mathcal{N}_B:B^\vee_0\rightarrow \mathbb{R}_+$
is real-analytic. To this end, it is equivalent to show that the composition
$\mathcal{N}_B\circ \overline{\Gamma}^{-1}:A^\vee_0\rightarrow \mathbb{R}_+$
is real-analytic. Again, let us simplify our discussion by biholomorphically
identifying $ A^\vee_0$
with $B^\vee_0$, and leaving $\overline{\Gamma}^{-1}$ implicit. Then 
fiberwise scalar multiplication on $B^\vee_0$ corresponds to the map
(\ref{eqn:bullet-action}). Thus if $(t,x)\in \mathbb{R}_+\times X'$ then
$\mathcal{N}_B(t\bullet x)=t^2$.

Let $\varUpsilon:\mathbb{R}_+\times X'\rightarrow {A^\vee_0}'$ be as in
Corollary \ref{cor:surjective-product-map}; then $\varUpsilon$ is a real-analytic
equivalence by Remark \ref{rem:real-analytic-equivalence},
and the previous remark implies that 
$$
\mathcal{N}_B\circ \varUpsilon:\mathbb{R}_+\times X'\rightarrow \mathbb{R}_+
$$
is real-analytic. Therefore so is 
$\mathcal{N}_B=\big(\mathcal{N}_B\circ \varUpsilon\big)
\circ \varUpsilon^{-1}$.

\end{proof}

We can consider the \textit{equivariant distortion function}
$K_k^\mu:M\rightarrow \mathbb{R}$ defined by setting
\begin{equation}
 \label{eqn:equivariant-distortion-fnct}
K_k^\mu(m)=:\Pi^\mu_k(x,x)=\sum_j\,\left|s_j^{(k)}(x)\right|^2,
\end{equation}
for $m\in M$ and any choice of 
$x\in \pi^{-1}(m)\subseteq X$, where the 
$s_j^{(k)}$'s are an orthonormal basis of $H^\mu_k(X)$ (see
(\ref{eqn:kernel-szego-equiv})). That $K_k^\mu$ is well-defined
follows from the fact that $\mu^X$ and $\nu^X$ commute (Lemma 2.1
of \cite{pao-IJM}). For any $m\in M$ and $t\in \mathbf{T}^1$,
given $x\in \pi^{-1}(m)$ by (\ref{eqn:isotypes}) we have
\begin{eqnarray*}
 K_k^\mu\left(\mu^M_{t^{-1}}(m)\right)&=&
\Pi^\mu_k\left(\mu^X_{t^{-1}}(x),\mu^X_{t^{-1}}(x)\right)=
\sum_j\,\left|s_j^{(k)}\left(\mu^X_{t^{-1}}(x)\right)\right|^2\\
&=&\sum_j\,\left|s_j^{(k)}(x)\right|^2=\Pi^\mu_k(x,x)=K_k^\mu(m).
\end{eqnarray*}

Therefore, $K_k^\mu\in \mathcal{C}^\infty(M)^\mu$; it may thus be regarded as a function
on $N'$ in a natural manner. We have in fact:

\begin{lem}
 \label{lem:analytic-distortion-function}
$K_k^\mu\in \mathcal{C}^\varpi(M)^\mu$. As a function on $N'$,
$K_k^\mu\in \mathcal{C}^\varpi(N')^\nu$.
\end{lem}

\begin{proof}
 By its very definition, $\Pi^\mu_k\in \mathcal{C}^\infty(X\times X)$
restricts to a sesquiholomorphic complex function on $A^\vee_0\times A^\vee_0$,
which is then \textit{a fortiori} real-analytic.
Since $X\times X$ is a real-analytic submanifold of $A^\vee_0\times A^\vee_0$
by Corollary \ref{cor:analytic-N-1}, we have
$\Pi^\mu_k\in \mathcal{C}^\varpi(X\times X)$. If now $\varphi$ is a local holomorphic
frame on a open subset $U\subset M$, the unitarization 
$\varphi_u=\varphi/\|\varphi\|_A:U\rightarrow X$ is real-analytic, where
$\|\varphi\|_A=:\big(\mathcal{N}_A\circ \varphi\big)^{1/2}$. Therefore,
$$
K^\mu_k(m)=\Pi^\mu_k\big(\varphi_u(m),\varphi_u(m)\big)
\,\,\,\,\,\,\,\,\,\,\,\,(m\in U)
$$
is real-analytic on $U$. The second statement is proved similarly (Lemma
\ref{lem:real-analytic-herm-metric-B}).
\end{proof}

\subsection{Asymptotics of sesqui-holomorphic extensions}
\label{sctn:sesqui-holomorphic-extensions}

Every $s\in H^\mu_k(X)$ extends uniquely to a holomorphic function
$\widetilde{s}:A^\vee_0\rightarrow \mathbb{C}$. Holomorphicity of the extended
action $\widetilde{\mu}^{A^\vee}$ implies, in view of (\ref{eqn:isotypes})
and (\ref{eqn:bullet-action}), that
for every $(z,\lambda)\in \mathbb{C}^*\times A^\vee_0$ we have
\begin{equation}
 \label{eqn:extension-equivariance}
\widetilde{s}(z\bullet \lambda)=z^k\,\widetilde{s}(\lambda).
\end{equation}

Given this and (\ref{eqn:kernel-szego-equiv}), we see that
$\Pi_k^\mu:X\times X\rightarrow \mathbb{C}$ extends uniquely to a sesquiholomorphic
function $\mathcal{P}^\mu_k:A^\vee_0\times A^\vee_0\rightarrow\mathbb{C}$, given by
\begin{equation}
 \label{eqn:kernel-szego-equiv-extended}
\mathcal{P}^\mu_k\left(\lambda,\lambda'\right)=
\sum_j\,\widetilde{s}_j^{(k)}(\lambda)\cdot \overline{\widetilde{s}_j^{(k)}(\lambda')}
\,\,\,\,\,\,\,\,\,\left(\lambda,\lambda'\in A^\vee_0\right),
\end{equation}
and satisfying, by (\ref{eqn:extension-equivariance}),
\begin{equation}
 \label{eqn:kernel-szego-equiv-extended-bullet}
\mathcal{P}^\mu_k\left(z\bullet\lambda,w\bullet\lambda'\right)=
z^k\,\overline{w}^k\,\mathcal{P}^\mu_k\left(\lambda,\lambda'\right),
\end{equation}
for every $z,\,w\in \mathbb{C}^*$.

Let $\sigma$ be a local holomorphic frame of $B^\vee$ on an open
subset $V\subseteq N'$. 
Then 
$$
\mathcal{P}^\mu_k\circ (\sigma\times \sigma):V\times V\rightarrow \mathbb{C},\,\,\,\,
(n,n')\mapsto \mathcal{P}^\mu_k\big(\sigma(n), \sigma(n')\big)
$$
is sesquiholomorphic. The unitarization
$\sigma_u=:(1/\|\sigma\|_B)\bullet\sigma:V\rightarrow X'$ (see Remark \ref{rem:natural-isomorphism_AB}) 
is a real-analytic section.
Given (\ref{eqn:kernel-szego-equiv-extended-bullet}), we have
\begin{eqnarray}
\label{eqn:key-relation-Pi-P}
 \lefteqn{\Pi^\mu_k\big(\sigma_u(n),\sigma_u(n')\big)=
\mathcal{P}^\mu_k\left(\dfrac{1}{\|\sigma(n)\|_B}\bullet \sigma (n),
\dfrac{1}{\|\sigma(n')\|_B}\bullet \sigma (n')\right)}\\
&=&\dfrac{1}{\|\sigma(n)\|_B^k}\,\dfrac{1}{\|\sigma(n')\|_B^k}\,
\mathcal{P}^\mu_k\left(\sigma (n),
\sigma (n')\right)=e^{-\frac k2\,\big( \varXi(n)+\varXi(n')\big)}\,
\mathcal{P}^\mu_k\left(\sigma (n),
\sigma (n')\right),\nonumber
\end{eqnarray}
where we have set, for $n\in V$,
\begin{equation}
 \label{eqn:kahler-potential-N}
\varXi(n)=\ln \left(\|\sigma(n)\|_B^2\right)=
\ln \Big(\ell_B\Big(\sigma(n),\sigma(n)\big)\Big).
\end{equation}
Then $\varXi$ is real-analytic by Lemma \ref{lem:real-analytic-herm-metric-B},
and furthermore $\partial_N\overline{\partial}_N\varXi=\Theta_B$,
where $\Theta_B=-2i\,\eta\in \Omega^2(N')$ is the curvature form of $B$.
In any given local coordinate chart $(z_k)$ for $N'$
this means that
$$
\dfrac{\partial^2\varXi}{\partial z_k\,\partial z_{\overline{l}}}
={\Theta_B}_{k\overline{l}}=-2i\,\eta_{k\overline{l}}
=2\,h_{k\overline{l}},
$$
where $h$ is the Riemannian metric of $(N',I,\eta)$. 
In other words, $\varXi$ is a K\"{a}hler potential for $2\,h$.

Being real-analytic, $\varXi$ has a unique sesquiholomorphic extension
$\widetilde{\varXi}$ to an open neighborhood of the diagonal 
$\widetilde{V}\subseteq V\times V$.
Similarly, by Lemma \ref{lem:analytic-distortion-function} $K^\mu_k$
also has a unique sesquiholomorphic extension $\widetilde{K^\mu_k}$
to an open neighborhood of the diagonal
in $N'\times N'$. 

\begin{lem}
 \label{lem:relation-sesquiholomorphic-extensions}
Let $\widetilde{V}\subseteq V\times V$ be an appropriate open neighborhood
of the diagonal. Then for every $(n,n')\in \widetilde{V}$ we have
\begin{equation}
 \label{eqn:relation-sesquiholomorphic-extensions}
\mathcal{P}^\mu_k\left(\sigma (n),\sigma (n')\right)=
e^{k\,\widetilde{\varXi}(n,n')}\,\widetilde{K^\mu_k}(n,n').
\end{equation}
\end{lem}

\begin{proof}
Both sides being sesquiholomorphic, it suffices to show that they have equal
restrictions on the diagonal. If $n=n'$, by (\ref{eqn:key-relation-Pi-P}) we have
\begin{eqnarray}
 \label{eqn:asymptotic-expansion-P-K-diagonal}
\mathcal{P}^\mu_k\left(\sigma (n),\sigma (n)\right)&=&
e^{k\,\varXi(n)}\,\Pi^\mu_k\big(\sigma_u(n),\sigma_u(n)\big)\\
&=&e^{k\,\widetilde{\varXi}(n,n)}\,K^\mu_k(n)=e^{k\,\widetilde{\varXi}(n,n)}\,
\widetilde{K^\mu_k}(n,n).\nonumber
\end{eqnarray}

\end{proof}

Inserting (\ref{eqn:relation-sesquiholomorphic-extensions}) 
in (\ref{eqn:key-relation-Pi-P}), we obtain for $(n,n')\in \widetilde{V}$:
\begin{equation}
 \label{eqn:key-relation-Pi-P-1}
\Pi^\mu_k\big(\sigma_u(n),\sigma_u(n')\big)=
e^{k\,\big [\widetilde{\varXi}(n,n')-\frac 12\,\varXi(n)-\frac 12\,\varXi(n')\big]}\,
\widetilde{K^\mu_k}(n,n').
\end{equation}

As discussed in the Introduction, by \cite{pao-IJM} if $m\in M'$ and
$x\in \pi^{-1}(m)$ there is an asymptotic expansion
(\ref{eqn:asymptotic-expansion-szego-special-case}), smoothly varying
on $M'$ and uniform on compact subsets of $M'$, with leading coefficient
$S_0^\mu=\Phi^{-(d+1)}$. Since $\Pi^\mu_k(x,x)$ 
is $\mu^M$-invariant, so is every $S_j^\mu$. Therefore, viewing $K^\mu_k$
as being defined on $N'$,
the expansion may be naturally
interpreted as holding on $N'$ (see (\ref{eqn:equivariant-distortion-fnct})
and Lemma \ref{lem:analytic-distortion-function})):
\begin{equation}
 \label{eqn:asymptotic-expansion-K-N}
K^\mu_k(n)\sim \left(\frac k\pi\right)^d\,\sum_{j\ge 0}k^{-j}\,S_j^\mu(n),
\end{equation}
where $S_0^\mu=\Phi^{-(d+1)}$. This suggest, heuristically, that
$\widetilde{K^\mu_k}(n,n')$ should satisfy a similar expansion, 
with coefficients the sesquiholomorphic extensions of the $S_j^\mu$'s.
This is indeed the case. 

To see this, let us consider first the asymptotics
of $\Pi^\mu_k\big(\sigma_u(n),\sigma_u(n')\big)$ for $(n,n')\in \widetilde{V}$.
Let 
\begin{equation}
 \label{eqn:szego-microlocal}
\Pi(x,y)=\int_0^{+\infty}e^{it\,\psi(x,y)}\,s(x,y,t)\,dt.
\end{equation}
be the usual Fourier integral representation of the Szeg\"{o} kernel
of $X$ determined in \cite{boutet-sjostraend}; here we think of $X$ as the 
unit circle bundle of $A^\vee$, with volume form $dV_X$. In particular,
$\Im(\psi)\ge 0$, and 
$s$ is a semiclassical symbol admitting an asymptotic expansion of the form
\begin{equation}
 \label{eqn:szego-microlocal-amplitude}
s(x,y,t)\sim \sum_{j\ge 0}t^{d-j}\,s_j(x,y)
\end{equation}
(see also the discussion 
in \cite{zelditch-theorem-of-Tian} and \cite{sz}).
For some $\epsilon>0$, let
$\varrho_1\in \mathcal{C}^\infty_0(-2\epsilon,2\epsilon)$ be a 
bump function identically equal to $1$ on $(-\epsilon,\epsilon)$.  
For some $C>0$,
let $\varrho_2\in \mathcal{C}^\infty_0\big(1/(2C),2C\big)$ be a bump function
identically equal to $1$ on $(1/C,C)$.
Let us write $\mu^X_{-\vartheta}$ for $\mu^X_{e^{-i\vartheta}}$.
Then, arguing as in the proof of Theorem 1 of 
\cite{pao-IJM},
\begin{eqnarray}
 \label{eqn:equivariant-szego-kernel-asymptotics}
\lefteqn{\Pi^\mu_k\big(\sigma_u(n),\sigma_u(n')\big)=
\dfrac{1}{2\pi}\,\int_{-\pi}^\pi\,e^{-ik\vartheta}\,\Pi\left(
\mu^X_{-\vartheta}\big(\sigma_u(n)\big),\sigma_u(n')\right)\,d\vartheta}\\
&\sim&\dfrac{1}{2\pi}\,\int_0^{+\infty}\int_{-\pi}^\pi\,e^{-ik\vartheta
+it\,\psi\left(
\mu^X_{-\vartheta}\big(\sigma_u(n)\big),\sigma_u(n')\right)}\,
s\left(
\mu^X_{-\vartheta}\big(\sigma_u(n)\big),\sigma_u(n'),t\right)\,\varrho_1(\vartheta)\,
dt\,d\vartheta\nonumber\\
&\sim&\dfrac{k}{2\pi}\,\int_0^{+\infty}\int_{-\pi}^\pi\,e^{ik\,\Psi(n,n',t,\vartheta)}\,
s\Big(
\mu^X_{-\vartheta}\big(\sigma_u(n)\big),\sigma_u(n'),kt\Big)\,\varrho_1(\vartheta)\,
\varrho_2(t)\,dt\,d\vartheta,\nonumber
\end{eqnarray}
where 
$$
\Psi(n,n',t,\vartheta)=:t\,\psi\Big(
\mu^X_{-\vartheta}\big(\sigma_u(n)\big),\sigma_u(n')\Big)-\vartheta.
$$
The last line of (\ref{eqn:equivariant-szego-kernel-asymptotics}) is an
oscillatory integral with phase $\Psi(n,n',t,\vartheta)$, and
$\Im(\psi)\ge 0$ implies $\Im(\Psi)\ge 0$.

Suppose first $n=n'$. Then one can see by (a slight adaptation of) the argument in
the proof Theorem 1 of \cite{pao-IJM} that the phase
$\Psi(n,n,t,\vartheta)$ has a unique stationary point
$P(n,n)=(t_0,\vartheta_0)=\big(1/\Phi(n),0\big)$, where as usual we think
of the invariant function $\Phi$ as descended on $N$.
Since $\psi(x,x)=0$ identically, we have $\Psi(n,n,t_0,\vartheta_0)=0$.
Furthermore, the Hessian matrix at $P_0$ is
$$
H_{P_0}(\Psi)=\left(\begin{matrix}
           0&\Phi(n)\\
\Phi(n)&\partial^2_{\vartheta\vartheta}\Psi(P_0)
          \end{matrix}\right).
$$
Therefore, $P_0$ is a non-degenerate critical point, and
by applying the stationary phase Lemma to it we obtain the
asymptotic expansion (\ref{eqn:asymptotic-expansion-K-N}).

By the theory of \cite{melin-sjostraend}, the stationary point and the 
asymptotic expansion will deform
smoothly with $(n,n')\in \widetilde{V}$, although the stationary point  
may cease to be real
when $n\neq n'$ (and should then be regarded as the stationary point
of an almost analytic extension of $\Psi$). More precisely, if 
$\widetilde{\Psi}\left(\widetilde{n},\widetilde{n}',\widetilde{t},\widetilde{\vartheta}\right)$ denotes
an almost analytic extension of $\Psi\left(n,n',t,\vartheta\right)$,
then the condition that 
$P\big(\widetilde{n},\widetilde{n}'\big)=
\left(\widetilde{t}\big(\widetilde{n},\widetilde{n}'\big),
\widetilde{\vartheta}\big(\widetilde{n},\widetilde{n}'\big)\right)$
be a stationary point of $\widetilde{\Psi}\left(\widetilde{n},\widetilde{n}',\cdot,\cdot\right)$
defines an almost analytic manifold $\left(\widetilde{t},
\widetilde{\vartheta}\right)=
\left(\widetilde{t}\big(\widetilde{n},\widetilde{n}'\big),
\widetilde{\vartheta}\big(\widetilde{n},\widetilde{n}'\big)\right)$. 

Applying to (\ref{eqn:equivariant-szego-kernel-asymptotics})  
the stationary phase Lemma for complex phase functions
from \S 2 of \cite{melin-sjostraend}
 for $(n,n')\in \widetilde{V}$ we obtain a smoothly varying asymptotic expansion
\begin{equation}
 \label{eqn:asymptotic-exp-Pi-n-n'}
\Pi^\mu_k\big(\sigma_u(n),\sigma_u(n')\big)\sim
\left(\frac k\pi\right)^d\,e^{ik\,\widetilde{\Psi}\big(n,n',P(n,n')\big)}\,
\,\sum_{j\ge 0}k^{-j}\,S_j(n,n'),
\end{equation}
for appropriate smooth functions $S_j(\cdot,\cdot)$ on $\widetilde{V}\subseteq V\times V$.

Given (\ref{eqn:key-relation-Pi-P}) and (\ref{eqn:asymptotic-exp-Pi-n-n'}), we get
\begin{eqnarray}
 \label{eqn:asympt-exp-P-k-n-n'}
\lefteqn{\mathcal{P}^\mu_k\left(\sigma (n),\sigma (n')\right)}\\
&\sim&e^{k\left[\frac 12\,\big( \varXi(n)+\varXi(n')\big)+i\,\widetilde{\Psi}\big(n,n',P(n,n')\big)\right]}\,
\left(\frac k\pi\right)^d\,\sum_{j\ge 0}k^{-j}\,S_j(n,n').\nonumber
\end{eqnarray}

Since the expansion holds in $\mathcal{C}^j$-norm for every $j$
and the left hand side is sesquiholomorphic, so is every term on the right hand side.
Therefore, each term
$$
e^{k\left[\frac 12\,\big( \varXi(n)+\varXi(n')\big)+i\,\widetilde{\Psi}\big(n,n',P(n,n')\big)\right]}\,
S_j(n,n')
$$
is the sesquiholomorphic extension of its diagonal restriction.

On the other hand, on the diagonal (\ref{eqn:asympt-exp-P-k-n-n'}) restricts to the
uniquely determined asymptotic expansion for 
(\ref{eqn:asymptotic-expansion-P-K-diagonal}),
and so we need to have
$S_j(n,n)=S_j^\mu(n)$, whence $S_j(n,n')=\widetilde{S_j^\mu}(n,n')$.
Furthermore, 
we see that
$$
\frac 12\,\big( \varXi(n)+\varXi(n')\big)+i\,\widetilde{\Psi}\big(n,n',P(n,n')\big)=
\widetilde{\varXi}(n,n').$$
Inserting this in (\ref{eqn:asymptotic-exp-Pi-n-n'}), we obtain
\begin{eqnarray}
 \label{eqn:asymptotic-expansion-Pi-k-n-n'-1}
\Pi^\mu_k\big(\sigma_u(n),\sigma_u(n')\big)&\sim&
\left(\frac k\pi\right)^d\,e^{k\,\left[
\widetilde{\varXi}(n,n')-\frac 12\,\big( \varXi(n)+\varXi(n')\big)\right]
}\nonumber\\
&&\cdot\sum_{j\ge 0}k^{-j}\,\widetilde{S^\mu_j}(n,n').
\end{eqnarray}
Now (\ref{eqn:key-relation-Pi-P-1}) and (\ref{eqn:key-relation-Pi-P})
imply
\begin{equation}
 \label{eqn:asymptotic-expansion-K-k-mu}
\widetilde{K^\mu_k}(n,n')\sim \left(\frac k\pi\right)^d\,\sum_{j\ge 0}k^{-j}\,\widetilde{S^\mu_j}(n,n')
\end{equation}
(see \cite{zelditch-theorem-of-Tian} and \cite{karabegov-schlichenmaier} for analogues 
in the standard case $\Phi=1$ ).

Analogous considerations hold for Toeplitz operators; see \S \ref{sctn:proof-heisenberg}.

\subsection{The Laplacian on invariant functions}
\label{sct:laplacian-invariant-functions}

Let us now dwell on the relation between the Laplacian operators
$\Delta_N$ and $\Delta_M$ of $(M,J,\omega)$ and $(N',I,\eta)$ 
acting on invariant functions. Thus let $f\in \mathcal{C}^\infty(M)^\mu$,
so that $f$ determines in a natural manner functions on $X$ and $N'$, respectively.
It is convenient in the present argument to explicitly distinguish the domain
of definition of the function in point, so we shall write $f=f_M$, and $f_X$ and $f_N$
to denote the induced functions on $X$ and $N'$, respectively.
It is also notationally convenient to leave $\overline{\Gamma}$ implicit,
and to identify $B^\vee_0$ with ${A^\vee_0}'$ (see Lemma \ref{lem:natural-isomorphism_AB}).
Thus we have holomorphic line bundle structures $\widehat{\pi}:A^\vee\rightarrow M$
and $\widehat{\kappa}:A^\vee\rightarrow N'$, where we write $\widehat{\kappa}$
for $\widehat{\kappa}\circ \overline{\Gamma}^{-1}$. The fibers of the latter are the orbits
of the complexified action $\widetilde{\mu}^{A^\vee}$.

Suppose $m\in M'$, $x\in \pi^{-1}(m)\in X\subset A^\vee_0$, and set $n=:\kappa(x)$.
Choose a local holomorphic frame $\varphi$ for $A^\vee$ on an open neighborhood
$U\subset M'$ of $m$, such that $\varphi(m)=x$ and which is horizontal at $m$, in the
sense of the proof of Proposition \ref{prop:real-analytic-eta}.
Then, as remarked in the same proof, $\varphi:U\rightarrow A^\vee$ is transverse at $m$
to the orbit of $\mu^X$ through $x$. 
In fact, in view of (\ref{eqn:second-infinitesimal-generator-line-bundle}),
$\varphi:U\rightarrow A^\vee$ is transverse at $m$
to the full orbit of $\widetilde{\mu}^{A^\vee}$ through $x$. 
Thus, the composition $\widehat{\kappa}\circ \varphi:U\rightarrow N'$
is holomorphic,
has maximal rank at $m$, and satisfies $\widehat{\kappa}\circ \varphi(m)=n$.
Therefore, there exists an open neighborhood $U\subseteq M'$ of $m$ such that
$V=:\widehat{\kappa}\circ \varphi(U)$ is open, and the induced map 
$\widehat{\kappa}\circ \varphi:U\rightarrow N'$ is a biholomorphism.

Let us set $Z=:\varphi(U)\subseteq A^\vee_0$.
Then $Z$ is a complex submanifold of $A^\vee_0$, and the restrictions of
$\widehat{\pi}$ and $\widehat{\kappa}$ to $Z$ determine biholomorphic maps
$\pi_Z:Z\rightarrow U$ and $\kappa_Z:Z\rightarrow V$.
The invariance hypothesis on $f$ implies that $f_M\circ \pi_Z=f_N\circ \kappa_Z$;
let us write $f_Z$ for this function. 

Furthermore, if $K$ is the complex structure on $Z$ then by holomorphicity
we can pull back the K\"{a}hler structures $(M,J,\omega)$ and 
$(N,I,\eta)$ under $\pi_Z$ and $\kappa_Z$,
respectively, to K\"{a}hler structures  $(Z,K,\omega')$ and $(Z,K,\eta')$.
Clearly 
$$
\Delta_M(f_M)\circ \pi_Z=\Delta_1(f_Z),\,\,\,\,\,\,\,\,\,
\Delta_N(f_N)\circ \kappa_Z=\Delta_2(f_Z),
$$
where $\Delta_1$ and $\Delta_2$ are the Laplacian operators
in the K\"{a}hler structures  $(Z,K,2\omega')$,
and $(Z,K,2\eta')$, respectively. Therefore,
\begin{equation}
 \label{eqn:laplacian-comparison-projection}
\Delta_M(f_M)(m)=\Delta_1(f_Z)(x),\,\,\,\,\,\,\,\,\,
\Delta_N(f_N)(n)=\Delta_2(f_Z)(x).
\end{equation}

Recall that $g(\cdot,\cdot)=\omega\big(\cdot,J(\cdot)\big)$ and
$h=\eta\big(\cdot,I(\cdot)\big)$ are the Riemannian metrics on $(M,J,\omega)$ and 
$(N,I,\eta)$, and by pull-back
view them as the Riemannian
metrics of $(Z,K,\omega')$ and $(Z,K,\eta')$, respectively.
Perhaps after restricting $U$ to a smaller open neighborhood of $m$ in $M'$,
we may assume without loss that on $Z$ there is a global
holomorphic coordinate chart $(z_j)$.
Let $g_{a\overline{b}}=g\big(\partial_a,\partial_{\overline{b}})$ and
$h_{a\overline{b}}=h\big(\partial_a,\partial_{\overline{b}})$
the respective covariant metric tensors, with associated contravariant 
tensors $\big(g^{\overline{b}a}\big)$ and $\big(h^{\overline{b}a}\big)$.

In particular,
$(T_xZ,K_x,\omega'_x)=(\mathcal{H}_x,J_{\mathcal{H},x},\omega_x)$,
where $\omega_x$ is $\omega_m$ pulled-back to $\mathcal{H}_x$ under $d_x\pi$.
Similarly, with the same abuse of language, 
$(T_xZ,K_x,\eta'_x)=(\mathcal{H}_x,J_{\mathcal{H},x},\eta_x)$.
By horizontality, expression
(\ref{eqn:differential-alpha-tilde}) for $\kappa^*(2\eta)$
implies that $\eta_x=\omega_x/\Phi(m)$.
Hence $h_{a\overline{b}}(x)=g_{a\overline{b}}(x)/\Phi(m)$, and so
$h^{\overline{b}a}(x)=\Phi(m)\,g^{\overline{b}a}(x)$. Thus
we conclude that
\begin{eqnarray}
 \label{eqn:laplacian-comparison}
\Delta_2(f_Z)(x)&=&\frac 12\,h^{\overline{b}a}(x)
\,\partial_a\,\partial_{\overline{b}}f_Z(x)\nonumber\\
&=&\frac{1}{2}\,\Phi(m)\,g^{\overline{b}a}(x)\,
\partial_a\,\partial_{\overline{b}}f_Z(x)=\Phi(m)\,\Delta_1(f_Z)(x).
\end{eqnarray}

Given (\ref{eqn:laplacian-comparison-projection}) and (\ref{eqn:laplacian-comparison}),
we conclude that
\begin{equation}
 \label{eqn:laplacian-comparison-final}
\Delta_N(f_N)(n)=\Phi(m)\,\Delta_M(f_M)(m).
\end{equation}

Interpreting $\Delta_M$ and $\Delta_N$ as endomorphisms of $\mathcal{C}^\infty(M')^\mu$,
we can restate (\ref{eqn:laplacian-comparison-final}) by writing
\begin{equation}
 \label{eqn:laplacian-comparison-final-global}
\Delta_N=\Phi\cdot \Delta_M.
\end{equation}

\subsection{$\mu$-adapted Heisenberg local coordinates}
\label{sctn:adapted-HLC}
As mentioned in the Introduction, Heisenberg local coordinates (HLC)
for $X$ centered at some $x\in X$ where defined
in \cite{sz}; it is in these local coordinates that
near-diagonal Szeg\"{o} kernel scaling asymptotics exhibit their universal nature.
While we refer to \cite{sz} for a detailed discussion,
let us recall that they consist in the choice of an adapted local coordinate
chart for $M$ centered at $m=\pi(x)$, interwining the unitary structure
on $T_{m}M$ with the standard one on $\mathbb{C}^d$, and a preferred local frame of $A^\vee$
on a neighborhood of $m$,
having a prescribed second order jet at $m$. 

Let $\mathfrak{x}:(-\pi,\pi)\times B_{2d}(\mathbf{0},\epsilon)\rightarrow X$,
$\mathfrak{x}(\theta,\mathbf{v})=x+(\theta,\mathbf{v})$, be
a system of HLC centered at $x$. Then
$\mathfrak{x}^*(dV_X)(\theta,\mathbf{0})=
(2\pi)^{-1}\,|d\theta|\, d\mathcal{L}(\mathbf{v})$,
where $d\mathcal{L}(\mathbf{v})$ is the Lebesgue measure
on $\mathbf{R}^{2d}$. For $\mathbf{v}\in B_{2d}(\mathbf{0},\epsilon)$,
let us set $x+\mathbf{v}=:\mathfrak{x}(0,\mathbf{v})$.

It is natural here to modify the previous prescription so as to incorporate $\mu^X$
into an \lq equivariant\rq\, HLC system. Namely, 
let us define 
$\mathfrak{y}':\mathbf{T}^1\times B_{2d}(\mathbf{0},\epsilon)\rightarrow
X$ by letting
\begin{equation}
 \label{eqn:adapted-HLC-0}
\mathfrak{y}'\left(e^{i\vartheta},\mathbf{w}\right)=:e^{i\vartheta}\bullet (x+\mathbf{w}).
\end{equation}
Working in coordinates on $\mathbf{T}^1$, this yields a map
$\mathfrak{y}:(-\pi,\pi)\times B_{2d}(\mathbf{0},\epsilon)\rightarrow
X$ by setting
\begin{equation}
 \label{eqn:adapted-HLC}
\mathfrak{y}\big(\vartheta,\mathbf{w})=:e^{i\vartheta}\bullet (x+\mathbf{w}).
\end{equation}
If $\mathbf{H}(m)\in \mathbf{R}^{2d}$ is the local coordinate expression
of $\xi_M(m)\in T_mM$ (viewed as a column vector) then the local HLC expression
of $\xi_X(x)$ is $\big(\mathbf{H}(m),-\Phi(m)\big)\in \mathbb{R}^{2d}\times \mathbb{R}$.
If $(\theta,\mathbf{v})\sim (0,\mathbf{0})$, then 
by (\ref{eqn:infinitesimal-lift}) 
\begin{equation}
 \label{eqn:local-coordinate-expression-eta}
\mathfrak{y}\big(\vartheta,\mathbf{w})=
x+\big(\mathbf{w}+\vartheta\,\mathbf{H}(m),-\vartheta\,\Phi(m)\big)+
O\left(\|(\mathbf{w},\vartheta)\|^2\right).
\end{equation}
The Jacobian matrix at the origin of $\mathfrak{x}^{-1}\circ \mathfrak{y}$ is then
\begin{equation}
 \label{eqn:jacobian-HLC-change}
\mathrm{Jac}_{(0,\mathbf{0})}\left(\mathfrak{x}^{-1}\circ \mathfrak{y}\right)=
\left(\begin{matrix}
       I_{2d}&\mathbf{H}_m\\
\mathbf{0}^t&-\Phi(m)
      \end{matrix}\right).
\end{equation}

Since $\Phi>0$, $\mathfrak{y}'$ is a local diffeomorphism at 
$(1,\mathbf{0})\in \mathbf{T}^1\times B_{2d}(\mathbf{0},\epsilon)$.
Therefore, if 
$T_\delta=:\left\{e^{i\vartheta}:-\delta<\vartheta<\delta\right\}\subseteq \mathbf{T}^1$
then for all sufficiently
small $\delta,\,\epsilon>0$,
the restriction of $\mathfrak{y}'$
to $T_\delta\times B_{2d}(\mathbf{0},\epsilon)$
is a diffeomorphism onto its image.

\begin{lem}
 \label{lem:injective-adapted-HLC}
Suppose $x\in X'$.
Then, for all sufficiently small $\epsilon>0$, the restriction 
$\mathfrak{y}':\mathbf{T}^1\times B_{2d}(\mathbf{0},\epsilon)\rightarrow X$ is injective.
Its image is a $\mu^X$-invariant tubular neighborhood of the $\mu^X$-orbit of
$x$.
\end{lem}

\begin{proof}
If not, there exists a sequence $\epsilon_j\rightarrow 0^+$ and
for every $j$ a choice of distinct pairs
$$
\left(e^{i\vartheta_j},\mathbf{w}_j\right),\,\,\,\, 
\left(e^{i\vartheta'_j},\mathbf{w}'_j\right)\in \mathbf{T}^1\times B_{2d}(\mathbf{0},\epsilon_j),
$$ and 
such that, if 
$\lambda_j=:\vartheta_j'-\vartheta_j$,
\begin{equation}
 \label{eqn:paradox-injective}
e^{i\vartheta_j}\bullet (x+\mathbf{w}_j)=e^{i\vartheta_j'}\bullet (x+\mathbf{w}'_j)\,\,\,\,\,\,
\Longrightarrow\,\,\,\,\,\,x+\mathbf{w}_j=e^{i\lambda_j}\bullet (x+\mathbf{w}'_j).
\end{equation}
If $e^{i\lambda_j}\in T_\delta$, the previous considerations imply
that $e^{i\lambda_j}=1$, whence $e^{i\vartheta_j}=e^{i\vartheta'_j}$, 
and $\mathbf{w}_j=\mathbf{w}'_j$, against the assumptions. 
Therefore, it follows from (\ref{eqn:paradox-injective}) that
$e^{i\lambda_j}\in \mathbf{T}^1\setminus T_\delta$, a compact subset of $\mathbf{T}^1$.
Perhaps after passing to a subsequence, we may therefore assume without loss that
$e^{i\lambda_j}\rightarrow e^{i\lambda_\infty}
\in \mathbf{T}^1\setminus T_\delta$ as $j\rightarrow +\infty$. 
Since obviously 
$x+\mathbf{w}_j,\,x+\mathbf{w}_j'\rightarrow x$ as $j\rightarrow +\infty$, passing to the limit
in (\ref{eqn:paradox-injective}) we obtain $e^{i\lambda_\infty}\bullet x=x$. But this is absurd
by definition of $X'$,
given that $e^{i\lambda_\infty}\neq 1$ and $x\in X'$.

\end{proof}

It follows easily that if $x\in X'$ then
$\mathfrak{y}':\mathbf{T}^1\times B_{2d}(\mathbf{0},\epsilon)\rightarrow X$ 
is a diffeomorphism onto its image for all sufficiently small $\epsilon>0$, and therefore that
$\mathfrak{y}:(-\pi,\pi)\times B_{2d}(\mathbf{0},\epsilon)\rightarrow
X$ is a local coordinate chart. We shall say that $\eta$ is a system of $\mu$-adapted HLC.

In general, $\mathfrak{y}':\mathbf{T}^1\times B_{2d}(\mathbf{0},\epsilon)\rightarrow X$
is a $l:1$-covering, where $l=|T_m|$. To see this, let us consider the following generalization
of Lemma \ref{lem:injective-adapted-HLC}:

\begin{lem}
 \label{lem:cover-adapted-HLC}
Suppose $l=|T_m|$, where $m=\pi(x)$. Then, for all sufficiently small $\epsilon>0$, the restriction 
$\mathfrak{y}':\mathbf{T}^1\times B_{2d}(\mathbf{0},\epsilon)\rightarrow X$ is 
an $l:1$-covering. Its image is a $\mu^X$-invariant tubular neighborhood of the $\mu^X$-orbit of
$x$.
\end{lem}

\begin{proof}
Suppose $x'=e^{i\vartheta_0}\bullet x\in \mathbf{T}^1\cdot x$.
Then for any $g\in T_m$ we have
$\mathfrak{y}\big((e^{i\vartheta_0}\,g,\mathbf{0})\big)=x'$. Therefore,
the inverse image ${\mathfrak{y}'}^{-1}(x')$ contains $l$ distinct elements
$(e^{i\vartheta_0}\,g,\mathbf{0})$ ($g\in T_m$), and at each of these
$\mathfrak{y}'$ is a local diffeomorphism. 
It follows that any $x''$ sufficiently close to the orbit
$\mathbf{T}^1\cdot x$ has at least $l$ inverse images under $\mathfrak{y}'$,
and that at each of these the latter is a local diffeomorphism.

I claim that in fact any $x''$ sufficiently close to the orbit
$\mathbf{T}^1\cdot x$ has exactly $l$ inverse images under $\mathfrak{y}'$.
If not, there exist $\epsilon_j\rightarrow 0^+$ and for every $j$
distinct pairs
$$\big(g_j^{(a)},\mathbf{v}_j^{(a)}\big)
\in \mathbf{T}^1\times B_{2d}(\mathbf{0},\epsilon),\,\,\,\,\,\,
1\le a\le l+1,$$
such that $g_j^{(a)}\bullet \mathbf{v}_j^{(a)}=g_j^{(b)}\bullet\mathbf{v}_j^{(b)}$,
for every $1\le a,b\le l+1$. 
Arguing as in the proof of Lemma \ref{lem:injective-adapted-HLC}, we conclude that
$g_j^{(a)}\,{g_j^{(b)}}^{-1}\not\in T_\delta$ for any $1\le b< a\le l+1$
and $j\gg 0$. In particular, 
perhaps after passing to a subsequence, for $a=2,\ldots,l+1$ we have
$g_j^{(a)}\,{g_j^{(1)}}^{-1}\rightarrow \lambda_\infty^{(a)}\in T_m\setminus T_\delta$.

Suppose $\lambda_\infty^{(a)}=\lambda_\infty^{(b)}$
for $2\le a< b\le l+1$. Then $g_j^{(a)}\,{g_j^{(b)}}^{-1}\rightarrow 1\in T_\delta$
as $j\rightarrow+\infty$, absurd. Therefore, $T_m$ contains the $l+1$
distinct elements
$\{1,\lambda_\infty^{(2)},\ldots,\lambda_\infty^{(l+1)}\}$, a contradiction.
\end{proof}

\begin{lem}
For any $\vartheta\in (-\pi,\pi)$, we have
$$\mathfrak{y}^*(dV_X)(\vartheta,\mathbf{0})=\dfrac{1}{2\pi}\,\Phi(m)\,|d\vartheta|\,
d\mathcal{L}(\mathbf{w}),$$
where $d\mathcal{L}(\mathbf{w})$ is the Lebesgue measure
on $\mathbf{R}^{2d}$. 
\end{lem}

\begin{proof}
Let us write 
$\mathfrak{x}^*(dV_X)=\mathcal{V}(\theta,\mathbf{v})\,|d\theta|\,d\mathcal{L}(\mathbf{v})$,
so that $\mathcal{V}(\theta,\mathbf{0})=(2\pi)^{-1}$. Then
\begin{eqnarray*}
 \mathfrak{y}^*(dV_X)&=&\left(\mathfrak{x}\circ \mathfrak{x}^{-1}\circ
\mathfrak{y}\right)^*\big(dV_X\big)=\left(\mathfrak{x}^{-1}\circ
\mathfrak{y}\right)^*\left(\mathfrak{x}^*\big(dV_X\big)\right)\\
&=&\left(\mathfrak{x}^{-1}\circ
\mathfrak{y}\right)^*\Big(
\mathcal{V}(\theta,\mathbf{v})\,d\theta\,d\mathcal{L}(\mathbf{v})\\
&=&
\Big(\mathcal{V}\circ \left(\mathfrak{x}^{-1}\circ
\mathfrak{y}\right)\Big)
\cdot 
\left|
\det 
\Big(
\mathrm{Jac}
\left(
\mathfrak{x}^{-1}\circ
\mathfrak{y}
\right)
\Big)
\right|\,|d\vartheta|\,
d\mathcal{L}(\mathbf{w}).
\end{eqnarray*}

At $(0,\mathbf{0})$, in view of (\ref{eqn:jacobian-HLC-change}) and since
$\Phi>0$, we get
\begin{equation}
 \label{eqn:jacobian-HLC-change-origin}
\mathfrak{y}^*(dV_X)(0,\mathbf{0})=
\dfrac{1}{2\pi}\,\Phi(m)\,|d\vartheta|\,
d\mathcal{L}(\mathbf{w}).
\end{equation}
This proves the claim at $(0,\mathbf{0})$. To prove it at $(\vartheta_0,\mathbf{0})$,
we replace $\vartheta\sim \vartheta_0$ by $\vartheta+\vartheta_0$ with $\vartheta\sim 0$
and note that $e^{i(\vartheta+\vartheta_0)}\bullet (x+\mathbf{v})=
e^{i\vartheta}\bullet \left(e^{i\vartheta_0}\bullet (x+\mathbf{v})\right)$.
Since $\mathfrak{x}_{\vartheta_0}(\theta,\mathbf{v})=
e^{i\theta}\cdot \left(e^{i\vartheta_0}\bullet (x+\mathbf{v})\right)$
is a system of HLC centered at $e^{i\vartheta_0}\bullet x$, one can argue as in the previous
case.
\end{proof}

\begin{cor}
 \label{cor:computation-integral}
Under the assumptions of Lemma \ref{lem:cover-adapted-HLC},
if $\epsilon>0$ is sufficiently small 
let $V_\epsilon=\mathfrak{y}'\big(\mathbf{T}^1\times B_{2d}(\mathbf{0},\epsilon)\big)$.
Then for any continuous function on $X$, we have
$$
\int_{V_\epsilon}\,f\,dV_X=\dfrac{1}{2\pi\,|T_m|}\,
\int_{-\pi}^\pi\,\int_{B_{2d}(\mathbf{0},\epsilon)}
\,f\circ \mathfrak{y}\cdot\big(\Phi(m)+A(\mathbf{w})\big)\,|d\vartheta|\,
d\mathcal{L}(\mathbf{w}),
$$
where $A(\mathbf{w})=O(\|\mathbf{w}\|)$.
\end{cor}

\section{Proof of Theorem \ref{thm:lower-order-terms}}

\begin{proof}
 Being the orthogonal projector 
$\Pi^\mu_k:L^2(X,dV_X)\rightarrow H^\mu_k(X)$,
$\Pi^\mu_k$ is idempotent; therefore for every $x\in X$ the 
Schwartz kernel 
$\Pi^\mu_k\in \mathcal{C}^\infty(X\times X)$
satisfies 
\begin{equation}
\label{eqn:composition-idempotent}
 \Pi_k^\mu(x,x)=\int_X\,\Pi^\mu_k(x,y)\,\Pi^\mu_k(y,x)\,dV_X(y).
\end{equation}

Let us fix $x_0\in X'$ and set $m_0=:\pi(x_0)\in M'$, $n_0=:\kappa(x_0)\in N'$,
and apply (\ref{eqn:composition-idempotent}) with $x=x_0$.
Let $\sigma$ be a local holomorphic frame of $B^\vee$ on an open neighborhood $V\subseteq N'$
of $n_0$; as usual we implicitly identify $B^\vee_0$ with $A^\vee_0$ by means of $\overline{\varGamma}$  
(Lemma \ref{lem:natural-isomorphism_AB}). We may assume without loss that $\sigma(n_0)=x_0$.
Let $\|\sigma\|_B=:(\mathcal{N}_B\circ \sigma)^{1/2}$. Then $\|\sigma\|_B$ is a positive real-analytic
function on $V$ by Lemma \ref{lem:real-analytic-herm-metric-B}. Therefore, the unitarization
$\sigma_u=:(1/\|\sigma\|_B)\bullet\sigma:V\rightarrow X'$ (see Remark \ref{rem:natural-isomorphism_AB}) 
is a real-analytic section and 
$\sigma_u(n_0)=x_0$.

There exists $\epsilon>0$ such that 
$\mathrm{dist}_X\left(x_0,\mathbf{T}^1\cdot y\right)\ge
\delta$ for every $y\in X\setminus \kappa^{-1}(V)$. 
Therefore, by Theorem 1 of \cite{pao-IJM}
we have $\Pi^\mu_k(x_0,\cdot)=O\left(k^{-\infty}\right)$ uniformly on $X\setminus \kappa^{-1}(V)$.
If $\sim$ stands for \lq has the same asymptotics as\rq, we see from this and
(\ref{eqn:composition-idempotent}) for $x=x_0$ that
\begin{equation}
\label{eqn:composition-idempotent-neighborhood}
 \Pi_k^\mu(x_0,x_0)\sim\int_{\kappa^{-1}(V)}\,\Pi^\mu_k(x_0,y)\,\Pi^\mu_k(y,x_0)\,dV_X(y).
\end{equation}

We can parametrize the invariant open neighborhood $\kappa^{-1}(V)\subseteq X'$ by setting
\begin{equation}
 \label{eqn:parametrization}
\varrho:\mathbf{T}^1\times V\rightarrow \kappa^{-1}(V),\,\,\,\,\,
\left(e^{i\vartheta},n\right)\mapsto e^{i\vartheta}\bullet \sigma_u(n).
\end{equation}
Then
\begin{equation}
 \label{eqn:volume-form-parametrization}
\varrho^*(dW_X)=\dfrac{1}{2\pi}\,d\vartheta\wedge dV_N
\end{equation}
where $dV_N=(1/d!)\,\eta^{\wedge d}$ (see \S \ref{sctn:circle-bundle-structure}).
Now (\ref{eqn:isotypes}) means that 
$s\big(e^{i\vartheta}\bullet x\big)=e^{ik\vartheta}\,s(x)$, for every 
$e^{i\vartheta}\in \mathbf{T}^1$ and $x\in X$. Therefore, given 
(\ref{eqn:kernel-szego-equiv}), we have
\begin{eqnarray*}
 \lefteqn{\Pi^\mu_k\left( x_0,e^{i\vartheta}\bullet \sigma_u(n)\right)\,
\Pi^\mu_k\left( e^{i\vartheta}\bullet \sigma_u(n),x\right)}\\
&=&
\left[e^{-ik\vartheta}\,\Pi^\mu_k\left( x_0,\sigma_u(n)\right)\right]\,
\left[e^{ik\vartheta}\,\Pi^\mu_k\left( \sigma_u(n),x_0\right)\right]\\
&=&
\Pi^\mu_k\left( x_0,\sigma_u(n)\right)\,
\Pi^\mu_k\left( \sigma_u(n),x_0\right).
\end{eqnarray*}

Inserting this and (\ref{eqn:comparison-volume-forms}) 
in (\ref{eqn:composition-idempotent-neighborhood}) we obtain
\begin{eqnarray}
 \label{eqn:composition-idempotent-parametrized}
\lefteqn{\Pi^\mu_k(x_0,x_0)}\nonumber\\
&\sim&\dfrac{1}{2\pi}\,\int_{-\pi}^\pi\,\int_V
\Pi^\mu_k\left( x_0,\sigma_u(n)\right)\,
\Pi^\mu_k\left( \sigma_u(n),x_0\right)\,\Phi(n)^{d+1}\,d\vartheta\, dV_N(n)\nonumber\\
&=&\int_V
\Pi^\mu_k\left(\sigma_u(n_0),\sigma_u(n)\right)\,
\Pi^\mu_k\left( \sigma_u(n),\sigma_u(n_0)\right)\,\Phi(n)^{d+1}\,dV_N(n).
\end{eqnarray}

If we use (\ref{eqn:key-relation-Pi-P-1}) in (\ref{eqn:composition-idempotent-parametrized})
we get
\begin{eqnarray}
 \label{eqn:composition-idempotent-parametrized-1}
\lefteqn{\Pi^\mu_k(x_0,x_0)}\nonumber\\
&\sim&\int_Ve^{-k\,\mathcal{D}_N(n_0,n)}\,\widetilde{K^\mu_k}(n_0,n)\,\widetilde{K^\mu_k}(n,n_0)
\,\Phi(n)^{d+1}\,dV_N(n),
\end{eqnarray}
where $\mathcal{D}$ is Calabi's diastasis function of $(N',I,2\eta)$, defined in 
(\ref{eqn:calabi-diastasis}).

Let us set, for simplicity, $\eta'=2\,\eta$. Also, suppose without loss
that $V$ is the domain of a holomorphic local coordinate chart
$(z_a)$ for $N'$. If $z_a+i\,y_a$, with $x_a,\,y_a$ real-valued, 
then by (\ref{eqn:riem-vol-element})
we have
\begin{eqnarray}
 \label{eqn:volume-element-N-2eta}
dV_N&=&\det \big([2\,\eta_{k\overline{l}}]\big)\cdot dx_1\wedge \cdots dx_d\wedge dy_1\cdots \wedge
dy_d\nonumber\\
&=&\det \big([\eta'_{k\overline{l}}]\big)\cdot dx_1\wedge \cdots dx_d\wedge dy_1\cdots \wedge
dy_d.
\end{eqnarray}
In view of (\ref{eqn:asymptotic-expansion-K-k-mu}),
we can thus rewrite (\ref{eqn:composition-idempotent-parametrized-1}) as follows:
\begin{eqnarray}
 \label{eqn:composition-idempotent-parametrized-2}
\lefteqn{\Pi^\mu_k(x_0,x_0)}\\
&\sim&\left(\frac k\pi\right)^{2d}\,\sum_{j\ge 0}k^{-j}\int_B
e^{-k\,\mathcal{D}_N(n_0,n)}\,Z_j(n_0,n)
\,\det \big([\eta'_{k\overline{l}}]\big)\,dx\,dy,\nonumber
\end{eqnarray}
where now $B\subseteq \mathbb{C}^d$ is some open ball centered at the origin, and for every
$j\ge 0$ we have
\begin{equation}
 \label{eqn:definition-Zj}
Z_j(n,n')=:\Phi(n')^{d+1}\,\sum_{a+b=j}\widetilde{S_a^\mu}(n,n')\,\widetilde{S_b^\mu}(n',n)\,\,\,\,\,\,\,
\big((n,n')\in V\times V\big).
\end{equation}
In particular, since $S_0^\mu=\Phi^{-(d+1)}$,
for $j=0$ we get from(\ref{eqn:definition-Zj}):
\begin{eqnarray}
 \label{eqn:definition-Z0}
Z_0(n,n')&=&\Phi(n')^{d+1}\,\widetilde{\Phi}(n,n')^{-(d+1)}\,\widetilde{\Phi}(n',n)^{-(d+1)}\\
&=&\left(\dfrac{\Phi(n')}{\widetilde{\Phi}(n,n')\,\widetilde{\Phi}(n',n)}\right)^{d+1}
=F_\Phi(n')^{d+1},
\nonumber
\end{eqnarray}
with the notation of Lemma \ref{lem:technical}, taking $p_0=n$, and
where $\widetilde{\Phi}$ is the sesquiholomorphic extension of $\Phi$ (as a function on
$N'$) to some open neighborhood $\widetilde{N}$
of the diagonal (and we assume $V\times V\subseteq \widetilde{N}$). On the diagonal,
$Z_0(n,n)=\Phi(n)^{-(d+1)}$.

On the other hand, for $j\ge 1$ we get
\begin{eqnarray}
 \label{eqn:definition-Zjge1}
Z_j(n,n')&=&\Phi(n')^{d+1}\,\left[\widetilde{\Phi}(n,n')^{-(d+1)}\,\widetilde{S_j^\mu}(n',n)+
\widetilde{S_j^\mu}(n,n')\,\widetilde{\Phi}(n',n)^{-(d+1)}\right]\nonumber\\
&&+\Phi(n')^{d+1}\,\sum_{0<a<j}\widetilde{S_a^\mu}(n,n')\,\widetilde{S_{j-a}^\mu}(n',n).
\end{eqnarray}
On the diagonal,
\begin{eqnarray}
 \label{eqn:diagonal-Zjge1}
Z_j(n,n)=2\,S_j^\mu(n)
+\Phi(n)^{d+1}\,\sum_{0<a<j}S_a^\mu(n)\,S_{j-a}^\mu(n).
\end{eqnarray}

Let us now consider the asymptotics of the $j$-th summand in 
(\ref{eqn:composition-idempotent-parametrized-2}).
Because $\mathcal{D}_N$ is the diastasis function of $\eta'$, we can apply
Theorem 3 of \cite{englis}, and obtain an asymptotic expansion
of the form
\begin{eqnarray}
 \label{eqn:asymptotic-expansion-Zj}
\lefteqn{\int_B
e^{-k\,\mathcal{D}_N(n_0,n)}\,Z_j(n_0,n)
\,\det \big([\eta'_{k\overline{l}}]\big)\,dx\,dy}\nonumber\\
&\sim&
\left(\frac\pi k\right)^d\,\sum_{l\ge 0}k^{-l}\,\left.R_l^N\big(Z_j(n_0,\cdot)\big)\right|_{n=n_0},
\end{eqnarray}
where the $R_j^N$'s are Englis' operators for the K\"{a}hler manifold $(N,I,\eta')$.

Using (\ref{eqn:asymptotic-expansion-Zj}) within (\ref{eqn:composition-idempotent-parametrized-2}),
we get 
\begin{eqnarray}
 \label{eqn:composition-idempotent-parametrized-3}
\Pi^\mu_k(x_0,x_0)
&\sim&\left(\frac k\pi\right)^{d}\,\sum_{j,l\ge 0}k^{-j-l}\,
\left.R_l^N\big(Z_j(n_0,\cdot)\big)\right|_{n=n_0}\nonumber\\
&=&\left(\frac k\pi\right)^{d}\,\sum_{j\ge 0}k^{-j}\,
\sum_{a+b=j}\left.R_a^N\big(Z_b(n_0,\cdot)\big)\right|_{n=n_0}.
\end{eqnarray}

It follows from (\ref{eqn:asymptotic-expansion-szego-special-case})
and (\ref{eqn:composition-idempotent-parametrized-3}) that 
\begin{eqnarray}
 \label{eqn:relation-for-S-j}
S^\mu_j(n_0)&=&\sum_{a+b=j}\left.R_a^N\big(Z_b(n_0,\cdot)\big)\right|_{n=n_0}\nonumber\\
&=&Z_j(n_0,n_0)+\sum_{a=1}^j\left.R_a^N\big(Z_{j-a}(n_0,\cdot)\big)\right|_{n=n_0}.
\end{eqnarray}
Given (\ref{eqn:diagonal-Zjge1}), the latter relation may be rewritten
\begin{eqnarray}
 \label{eqn:relation-for-S-j-1}
S^\mu_j(n_0)&=&2\,S_j^\mu(n_0)
+\Phi(n_0)^{d+1}\,\sum_{0<a<j}S_a^\mu(n_0)\,S_{j-a}^\mu(n_0)\\
&&
+\sum_{a=1}^j\left.R_a^N\big(Z_{j-a}(n_0,\cdot)\big)\right|_{n=n_0}.\nonumber
\end{eqnarray}
It follows that 
\begin{eqnarray}
 \label{eqn:relation-for-S-j-2}
S^\mu_j(n_0)=-\Phi(n_0)^{d+1}\,\sum_{0<a<j}S_a^\mu(n_0)\,S_{j-a}^\mu(n_0)
-\sum_{a=1}^j\left.R_a^N\big(Z_{j-a}(n_0,\cdot)\big)\right|_{n=n_0},
\end{eqnarray}
which determines $S^\mu_j$ for any $j\ge 1$ in terms of the $S^\mu_k$'s with
$0\le k<j$ and their sesquiholomorphic extensions. The proof is complete, 
for (\ref{eqn:relation-for-S-j-2})
is (\ref{eqn:S-j-general}), with $j$ in place of $j+1$.
\end{proof}

\section{Proof of Corollary \ref{cor:lower-order-terms}}

\begin{proof}
 Let us apply (\ref{eqn:relation-for-S-j-2}) with $j=1$. We get
\begin{eqnarray}
 \label{eqn:relation-for-S-j=0}
S^\mu_1(n_0)&=&-\left.R_1^N\big(Z_0(n_0,\cdot)\big)\right|_{n=n_0}\nonumber\\
&=&-\left.\left(\Delta_N-\frac 12\,\varrho_N\right)\big(Z_0(n_0,\cdot)\big)\right|_{n=n_0}.
\end{eqnarray}
where $Z_0$ is defined by (\ref{eqn:definition-Z0}), and $\Delta_N$ and $\varrho_N$
are defined by (\ref{eqn:laplace-beltrami}) and (\ref{eqn:scalar-curvature}),
respectively, with reference to the K\"{a}hler manifold $(P,K,\gamma)=(N,I,\eta')$, where
$\eta'=2\,\eta$. 

We have, by (\ref{eqn:definition-Z0}),
\begin{equation}
 \label{eqn:relation-z0-power}
Z_0(n_0,n)=\left(\dfrac{\Phi(n)}{\widetilde{\Phi}(n_0,n)\,\widetilde{\Phi}(n,n_0)}\right)^{d+1}=
F_\Phi(n)^{d+1},
\end{equation}
where $F_\Phi$ is defined as in  Lemma \ref{lem:technical}, with $f=\Phi$ and 
$n_0=p_0$.
Applying (\ref{eqn:laplacian-powers-inductive}) with $l=d+1$ and $f=F_\Phi$, we get
\begin{eqnarray}
 \label{eqn:laplacian-powers-Phi}
\lefteqn{\left.\Delta_N\big(Z_0(n_0,\cdot)\big)\right|_{n=n_0}=\Delta_N\left(F_\Phi^{d+1}\right)(n_0)}
\nonumber\\
&=&(d+1)\,F_\Phi(n_0)^d\cdot \Delta_N(F_\Phi)(n_0)
+\dfrac{d(d+1)}{2}\,F_\Phi(n_0)^{d-1}\,
\big\|\mathrm{grad}_N(F_\Phi)(n_0)\big\|^2\nonumber\\
&=&(d+1)\,\Phi(n_0)^{-d}\cdot \Delta_N(F_\Phi)(n_0)
\end{eqnarray}
where the gradient and the norm are taken 
with respect to the Riemannian metric $h'=2\,h$, and in the last equation
we have made use of Lemma \ref{lem:technical-1}.

Let us apply Lemma \ref{lem:technical} with $(P,K,\gamma)=(N',I,\eta')$,
$f=\Phi^{d+1}\in \mathcal{C}^\varpi(N')$, and $p_0=n_0$, so that in the statement
we have $F=Z_0(n_0,\cdot)$.
We obtain
\begin{eqnarray}
 \label{eqn:laplacian-sesquiholomorphic-Phi}
\lefteqn{\left.\Delta_N\big(Z_0(n_0,\cdot)\big)\right|_{n=n_0}}
\\
&=&(d+1)\,\Phi(n_0)^{-(d+2)}\,\left[\Delta_P(\Phi)(n_0)
-\dfrac{1}{2\,\Phi(n_0)}\,\big\|\mathrm{grad}_N(\Phi)(n_0)\big\|^2\right].\nonumber
\end{eqnarray}
Inserting (\ref{eqn:laplacian-sesquiholomorphic-Phi}) 
in (\ref{eqn:relation-for-S-j=0})
\begin{eqnarray*}
S_1^\mu(n_0)&=&\frac 12\,\varrho_N(n_0)\,\Phi(n_0)^{-(d+1)}\\
&&+(d+1)\,\Phi(n_0)^{-(d+2)}\,\left[
\dfrac{1}{2\,\Phi(n_0)}\,\big\|\mathrm{grad}_N(\Phi)(n_0)\big\|^2
-\Delta_P(\Phi)(n_0)
\right].\nonumber
\end{eqnarray*}

\end{proof}

\section{Proof of Theorem \ref{thm:toeplitz-scaling}}

\begin{proof}
Statements 1. and 2. follow quite straightforwardly by using the corresponding
properties of $\Pi^\mu_k$ in Theorem 1 of \cite{pao-IJM}
in the first line of (\ref{eqn:kernel-topelitz-equiv}).

To prove 3., we start from the relation
\begin{eqnarray}
 \label{eqn:composition-kernel-scaling}
\lefteqn{
T^\mu_k(f)\left(x+\dfrac{\mathbf{v}}{\sqrt{k}},x+\dfrac{\mathbf{w}}{\sqrt{k}}\right)
}\\
&=&\int_X\Pi^\mu_k\left(x+\dfrac{\mathbf{v}}{\sqrt{k}},y\right)\,f(y)\,
\Pi^\mu_k\left(y,x+\dfrac{\mathbf{w}}{\sqrt{k}}\right)\,dV_X(y).\nonumber
\end{eqnarray}

If integration in $dV_X(y)$ in (\ref{eqn:composition-kernel-scaling}) is restricted to
a given invariant tubular neighborhood $V$ of the orbit $\mathbf{T}^1\cdot x$, only
a negligible contribution to the asymptotics is lost. On the other hand, on $V$ we can 
introduce $\mu$-adapted HLC as in \S \ref{sctn:adapted-HLC}, so as to write
$y=e^{i\theta}\bullet (x+\mathbf{u})$. Applying Corollary \ref{cor:computation-integral}
(with $V=V_\epsilon$), we get
\begin{eqnarray}
 \label{eqn:composition-kernel-scaling-0}
\lefteqn{
T^\mu_k(f)\left(x+\dfrac{\mathbf{v}}{\sqrt{k}},x+\dfrac{\mathbf{w}}{\sqrt{k}}\right)
}\\
&\sim&\dfrac{1}{2\pi\,|T_m|}\,
\int_{-\pi}^\pi\,\int_{B_{2d}(\mathbf{0},\epsilon)}
\,\big(\Phi(m)+A(\mathbf{u})\big)\nonumber\\
&&\cdot \Pi^\mu_k\left(x+\dfrac{\mathbf{v}}{\sqrt{k}},e^{i\theta}\bullet (x+\mathbf{u})\right)\,
f(m+\mathbf{u})\,
\Pi^\mu_k\left(e^{i\theta}\bullet (x+\mathbf{u}),x+\dfrac{\mathbf{w}}{\sqrt{k}}\right)\nonumber\\
&&\cdot|d\vartheta|\,
d\mathcal{L}(\mathbf{u}),\nonumber
\end{eqnarray}
where we used that $f\in \mathcal{C}^\infty(M)^\mu$.

Let $D_1,\,D_2>0$ be as in (\ref{eqn:bound-orbit-distance-HLC-first}).
Since $\|\mathbf{v}\|,\,\|\mathbf{w}\|\le C\,k^{1/9}$, we have
\begin{equation}
 \label{eqn:bound-on-orbit-distance}
\mathrm{dist}_X\left(\mathbf{T}^1\cdot x,x+\dfrac{\mathbf{v}}{\sqrt{k}}\right)
\le D_2\,C\,k^{-7/18}.
\end{equation}
If $\mathrm{dist}_X\left(\mathbf{T}^1\cdot x,y\right)\ge 2\,D_2\,C\,k^{-7/18}$, then by
(\ref{eqn:bound-on-orbit-distance}) we have
\begin{equation*}
\mathrm{dist}_X\left(\mathbf{T}^1\cdot y,x+\dfrac{\mathbf{v}}{\sqrt{k}}\right)\ge
D_2\,C\,k^{-7/18},
\end{equation*}
and similarly for $\mathbf{w}$. It follows from this and statement 2.
(with $\epsilon=1/9$) that the contribution to
(\ref{eqn:composition-kernel-scaling}) and (\ref{eqn:composition-kernel-scaling-0})
coming from the locus where 
$\mathrm{dist}_X\left(\mathbf{T}^1\cdot x,y\right)\ge 2\,D_2\,C\,k^{-7/18}$ is rapidly decreasing.
By (\ref{eqn:bound-orbit-distance-HLC-first}), this means that in (\ref{eqn:composition-kernel-scaling-0})
the contribution of 
the locus where $\|\mathbf{u}\|\ge (2D_2/D_1)\,C\,k^{-7/18}$ is rapidly decreasing.
Therefore, only a negligible contribution is lost in (\ref{eqn:composition-kernel-scaling-0})
if the integrand is multiplied by $\varrho\left(k^{7/18}\,\mathbf{w}\right)$, where $\varrho$
is an appropriate radial bump function, identically equal to $1$ near the origin.

Furthermore, using (\ref{eqn:isotypes}) and (\ref{eqn:bullet-action}), for 
any $x,x',x''\in X$ and $e^{i\theta}\in \mathbf{T}^1$ we have
$$
\Pi^\mu_k\left(x',e^{i\theta}\bullet x''\right)
=e^{-ik\theta}\,\Pi^\mu_k\left(x',x''\right)=\overline{\Pi^\mu_k\left(e^{i\theta}\bullet x'',x'\right)}.
$$

Inserting this in (\ref{eqn:composition-kernel-scaling-0}), and applying the rescaling
$\mathbf{u}\mapsto \mathbf{u}/\sqrt{k}$, we obtain
\begin{eqnarray}
 \label{eqn:composition-kernel-scaling-1}
\lefteqn{
T^\mu_k(f)\left(x+\dfrac{\mathbf{v}}{\sqrt{k}},x+\dfrac{\mathbf{w}}{\sqrt{k}}\right)
}\\
&\sim&\dfrac{k^{-d}}{|T_m|}\,
\int_{\mathbb{C}^d}
\,\left(\Phi(m)+A\left(\dfrac{\mathbf{u}}{\sqrt{k}}\right)\right)\nonumber\\
&&\cdot \Pi^\mu_k\left(x+\dfrac{\mathbf{v}}{\sqrt{k}},x+\dfrac{\mathbf{u}}{\sqrt{k}}\right)\,
f\left(m+\dfrac{\mathbf{u}}{\sqrt{k}}\right)\,
\Pi^\mu_k\left(x+\dfrac{\mathbf{u}}{\sqrt{k}},x+\dfrac{\mathbf{w}}{\sqrt{k}}\right)\nonumber\\
&&\cdot\varrho\left(k^{-1/9}\,\mathbf{u}\right)\,
d\mathcal{L}(\mathbf{u});\nonumber
\end{eqnarray}
integration in $d\mathbf{u}$ is really over an expanding 
ball of radius $O\big(k^{1/9}\big)$ in $\mathbb{C}^d$.

Now by (3) of
Theorem 1 of \cite{pao-IJM} (and the remark immediately following
the statement of that Theorem) with $\upsilon_1=(0,\mathbf{v})$
and $\upsilon_2=(0,\mathbf{w})$, the sought expansion
holds for $\Pi_k^\mu$ (that is, for $f=1$). Thus
\begin{eqnarray}
 \label{eqn:asymptotic-expansion-szego-equivariant}
\Pi^\mu_k\left(x+\dfrac{\mathbf{v}}{\sqrt{k}},x+\dfrac{\mathbf{u}}{\sqrt{k}}\right)
&\sim& \left(\frac k\pi\right)^d\cdot\sum_{t\in T_m}t^k\,
e^{\psi_2\left(d_m\mu^M_{t^{-1}}(\mathbf{v}),\mathbf{u}\right)/\Phi(m)} \\
&&\cdot
  \left(\Phi(m)^{-(d+1)}+
\sum _{j\ge 1}k^{-j/2}\,R_j\left(m,d_m\mu^M_{t^{-1}}(\mathbf{v}),\mathbf{u}\right) \right),  \nonumber
\end{eqnarray}
where $\psi_2$ is as in (\ref{eqn:definition-psi-2}), and 
$R_j(m,\mathbf{v},\mathbf{u})$ is a polynomial function of $\mathbf{v}$ and $\mathbf{u}$.
Clearly, 
$$
\dfrac{1}{\Phi(m)}\,\psi_2(\mathbf{v},\mathbf{u})=\psi_2\left(\dfrac{1}{\sqrt{\Phi(m)}}\,\mathbf{v},
\dfrac{1}{\sqrt{\Phi(m)}}\,\mathbf{u}\right)=\psi_2\left(\mathbf{v}',
\mathbf{u}'\right),
$$
where for any $\mathbf{p}\in \mathbb{C}^d$ we set $\mathbf{p}'=\mathbf{p}/\sqrt{\Phi(m)}$.

Using this and the Taylor expansion for $f(m+\mathbf{u}/\sqrt{k})$ at $m$, we get
for (\ref{eqn:composition-kernel-scaling-1}) an asymptotic expansion in descending powers of 
$k^{1/2}$, whose leading term is given by
\begin{eqnarray}
 \label{eqn:leading-term-toeplitz-scaling}
\lefteqn{\dfrac{k^{-d}}{|T_m|}\,\Phi(m)^{-2d-1}	,f(m)\,\left(\frac k\pi\right)^{2d}}\\
&&\cdot\sum_{t,s\in T_m}(s\,t)^k\,
\int_{\mathbb{C}^d}e^{\psi_2\left(d_m\mu^M_{t^{-1}}(\mathbf{v}'),\mathbf{u}'\right)+
\psi_2\left(\mathbf{u}',d_m\mu^M_{s}(\mathbf{w}')\right)}\,d\mathcal{L}(\mathbf{u}).
\nonumber
\end{eqnarray}
Applying the change of variable $\mathbf{u}=\sqrt{\Phi(m)}\,\mathbf{s}$, 
(\ref{eqn:leading-term-toeplitz-scaling}) becomes
\begin{eqnarray}
 \label{eqn:leading-term-toeplitz-scaling-0}
\lefteqn{\dfrac{1}{|T_m|}\,\Phi(m)^{-(d+1)}\,f(m)\,\frac{k^d}{\pi^{2d}}}\\
&&\cdot\sum_{t,s\in T_m}(s\,t)^k\,
\int_{\mathbb{C}^d}e^{\psi_2\left(d_m\mu^M_{t^{-1}}(\mathbf{v}'),\mathbf{s}\right)+
\psi_2\left(\mathbf{s},d_m\mu^M_{s}(\mathbf{w}')\right)}\,d\mathcal{L}(\mathbf{s}).
\nonumber\\
&=&\dfrac{1}{|T_m|}\,\Phi(m)^{-(d+1)}\,f(m)\,\frac{k^d}{\pi^{2d}}\,\pi^d\,
\sum_{t,s\in T_m}(s\,t)^k\,e^{\psi_2\left(d_m\mu^M_{t^{-1}}(\mathbf{v}'),d_m\mu^M_{s}(\mathbf{w}')\right)}
\nonumber\\
&=&\dfrac{1}{|T_m|}\,\Phi(m)^{-(d+1)}\,f(m)\,\left(\frac{k}{\pi}\right)^d\,
\sum_{t,s\in T_m}(s\,t)^k\,e^{\psi_2\left(d_m\mu^M_{(st)^{-1}}(\mathbf{v}'),\mathbf{w}'\right)}\nonumber\\
&=&\Phi(m)^{-(d+1)}\,f(m)\,\left(\frac{k}{\pi}\right)^d\,
\sum_{t\in T_m}t^k\,e^{\psi_2\left(d_m\mu^M_{t^{-1}}(\mathbf{v}'),\mathbf{w}'\right)}\nonumber\\
&=&\Phi(m)^{-(d+1)}\,f(m)\,\left(\frac{k}{\pi}\right)^d\,
\sum_{t\in T_m}t^k\,e^{\psi_2\left(d_m\mu^M_{t^{-1}}(\mathbf{v}),\mathbf{w}\right)/\Phi(m)}.\nonumber
\end{eqnarray}
We have used that if $A:\mathbb{C}^d\rightarrow \mathbb{C}^d$ is unitary,
then $\psi_2(\mathbf{u},A\mathbf{t})=\psi_2\left(A^{-1}\mathbf{u},\mathbf{t}\right)$ for any
$\mathbf{u},\,\mathbf{t}\in \mathbb{C}^d$, and the relation
$$
\int_{\mathbb{C}^d}e^{\psi_2(\mathbf{v},\mathbf{u})+\psi_2(\mathbf{u},\mathbf{w})}\,
d\mathcal{L}(\mathbf{u})=\pi^d\,e^{\psi_2(\mathbf{v},\mathbf{w})}.
$$

Finally, when $\mathbf{v}=\mathbf{w}=\mathbf{0}$ the appearance of descending powers
of $k$ in the asymptotic expansion for (\ref{eqn:composition-kernel-scaling-1}) originates from
Taylor expanding the integrand in $\mathbf{u}/\sqrt{k}$; half-integer powers of $k$
are thus associated to odd homogeneous polynomials in $\mathbf{u}$, and therefore the corresponding
contributions to the integral vanish by parity considerations.
\end{proof}

\section{Proof of Theorem \ref{thm:toeplitz-lower-order-terms}}

\begin{proof}
 The proof of Theorem is an adaptation of the one of Theorem \ref{thm:lower-order-terms},
so we'll be very sketchy.
Adopting the same set-up, rather than (\ref{eqn:composition-idempotent}),
(\ref{eqn:composition-idempotent-neighborhood}) and
(\ref{eqn:composition-idempotent-parametrized-1}) we now have
\begin{eqnarray}
\label{eqn:toeplitz-composition-integral}
 T^\mu_k[f]\left(x_0,x_0\right)&=&\int_{X}\,\Pi^\mu_k(x_0,y)\,f(y)\,\Pi^\mu_k(y,x_0)\,dV_X(y)\\
&\sim&\int_{\kappa^{-1}(V)}\,\Pi^\mu_k(x_0,y)\,f(y)\,\Pi^\mu_k(y,x_0)\,dV_X(y)\nonumber\\
&=&
\int_Ve^{-k\,\mathcal{D}_N(n_0,n)}\,\widetilde{K^\mu_k}(n_0,n)\,\widetilde{K^\mu_k}(n,n_0)
\,f(n)\,\Phi(n)^{d+1}\,dV_N(n)\nonumber
\end{eqnarray}
Therefore, we get in place of (\ref{eqn:composition-idempotent-parametrized-2})
and (\ref{eqn:composition-idempotent-parametrized-3}):
\begin{eqnarray}
 \label{eqn:toeplitz-composition-idempotent-parametrized-2}
\lefteqn{T^\mu_k[f]\left(x_0,x_0\right)}\nonumber\\
&\sim&\left(\frac k\pi\right)^{2d}\,\sum_{j\ge 0}k^{-j}\int_B
e^{-k\,\mathcal{D}_N(n_0,n)}\,Z_j(n_0,n)\,f(n)
\,\det \big([\eta'_{k\overline{l}}]\big)\,dx\,dy\nonumber\\
&\sim&\left(\frac k\pi\right)^{d}\,\sum_{j\ge 0}k^{-j}\,
\sum_{a+b=j}\left.R_a^N\Big(Z_b(n_0,\cdot)\,f(\cdot)\Big)\right|_{n=n_0},
\end{eqnarray}
which proves the claim (and reproves Corollary \ref{cor:toeplitz-scaling}).

\end{proof}

\section{Proof of Corollary \ref{cor:toeplitz-lower-order-terms}}

\begin{proof}
Let us simplify notation in the following arguments by setting $f^\mu_j=:S_j^\mu[f]$. 
To begin with, we have from (\ref{eqn:definition-Zjge1}) that
$Z_1(n_0)=2\,S_1^\mu(n_0)$.
We see from (\ref{eqn:toeplitz-composition-idempotent-parametrized-2}) that
\begin{eqnarray}
 \label{eqn:expression-for-f-1-mu}
f_1^\mu(n_0)&=&\left.R_0^N\Big(Z_1(n_0,\cdot)\,f(\cdot)\Big)\right|_{n=n_0}+
\left.R_1^N\Big(Z_0(n_0,\cdot)\,f(\cdot)\Big)\right|_{n=n_0}\\
&=&Z_1(n_0,n_0)\,f(n_0)+
\left.\left(\Delta_N-\frac 12\,\varrho_N\right)\Big(Z_0(n_0,\cdot)\,f(\cdot)\Big)\right|_{n=n_0}\nonumber\\
&=&\left[2\,S_1^\mu(n_0)-\frac 12\,\varrho_N(n_0)\,\Phi (n_0)^{-(d+1)}\right]\,f(n_0)+
\left.\Delta_N\Big(Z_0(n_0,\cdot)\,f(\cdot)\Big)\right|_{n=n_0}.\nonumber
\end{eqnarray}

Now in view of Lemma \ref{lem:technical-1} we have
\begin{eqnarray*}
 \lefteqn{\left.\Delta_N\Big(Z_0(n_0,\cdot)\,f(\cdot)\Big)\right|_{n=n_0}}\nonumber\\
&=&\left.\Delta_N\Big(Z_0(n_0,\cdot)\Big)\right|_{n=n_0}\, f(n_0)+
\Phi (n_0)^{-(d+1)}\,\left.\Delta_N\Big(f(\cdot)\Big)\right|_{n=n_0}.
\end{eqnarray*}
Inserting this in (\ref{eqn:expression-for-f-1-mu}), and recalling (\ref{eqn:relation-for-S-j=0}),  
we obtain
\begin{eqnarray}
 \label{eqn:expression-for-f-1-mu-0}
f_1^\mu(n_0)&=&\Phi (n_0)^{-(d+1)}\,\left.\Delta_N\Big(f(\cdot)\Big)\right|_{n=n_0}+
2\,S_1^\mu(n_0)\,f(n_0)\nonumber\\
&&+\left(-\frac 12\,\varrho_N(n_0)\,\Phi (n_0)^{-(d+1)}
+\left.\Delta_N\Big(Z_0(n_0,\cdot)\Big)\right|_{n=n_0}\right)\,f(n_0)\nonumber\\
&=&\Phi (n_0)^{-(d+1)}\,\left.\Delta_N\Big(f(\cdot)\Big)\right|_{n=n_0}+
S_1^\mu(n_0)\,f(n_0).
\end{eqnarray}
\end{proof}

\section{Proof of Corollary \ref{cor:berezin-transform}}

\begin{proof}
 Notation being as in Definition \ref{defn:equivariant-berezin-transform} and the proof
of Corollary \ref{cor:toeplitz-lower-order-terms}, by Corollaries
\ref{cor:lower-order-terms} and \ref{cor:toeplitz-lower-order-terms} we have
on $M'$
\begin{eqnarray*}
 \mathrm{Ber}^\mu_k[f]&=&
\dfrac{f_0^\mu+k^{-1}\,f_1^\mu+O\left(k^{-2}\right)}{S_0^\mu+k^{-1}\,S_1^\mu+O\left(k^{-2}\right)}
=\dfrac{f_0^\mu}{S_0^\mu}\cdot
\dfrac{1+k^{-1}\,(f_1^\mu/f^\mu_0)+O\left(k^{-2}\right)}{1+k^{-1}\,(S_1^\mu/S^\mu_0)+O\left(k^{-2}\right)}\\
&=&f+k^{-1}\,f\cdot \left(\dfrac{f_1^\mu}{f_0^\mu}
-\dfrac{S_1^\mu}{S_0^\mu}\right)+O\left(k^{-2}\right).
\end{eqnarray*}
Thus $B_0^\mu(f)=f$; furthermore, by Corollary \ref{cor:toeplitz-lower-order-terms}
we have
\begin{eqnarray*}
 B^\mu_1(f)&=&f\cdot \left(\dfrac{f_1^\mu}{f_0^\mu}
-\dfrac{S_1^\mu}{S_0^\mu}\right)=\Phi^{d+1}\,f_1^\mu-f\,\Phi^{d+1}\,S_1^\mu\\
&=&\Delta_N(f)+\Phi^{d+1}\,S_1^\mu\cdot f-f\cdot\Phi^{d+1}\,S_1^\mu=\Delta_N(f).
\end{eqnarray*}
\end{proof}

\section{Proof of Theorem \ref{thm:heisenberg-relation}}
\label{sctn:proof-heisenberg}

Before tackling the proof, let us remark that
considerations similar to those in \S \ref{sctn:sesqui-holomorphic-extensions}
hold for Toeplitz operators. Namely,
if $f\in \mathcal{C}^\infty(M)^\mu$ let $T^\mu_k[f]:H^\mu_k(X)\rightarrow H^\mu_k(X)$ 
and $T^\mu_k[f]\in \mathcal{C}^\infty(X\times X)$ denote both the induced operator and
its Schwartz kernel, given by (\ref{eqn:kernel-topelitz-equiv}). 
The latter extends uniquely to a sesquiholomorphic function
$\mathcal{T}^\mu_k[f]:A^\vee_0\times A^\vee_0\rightarrow\mathbb{C}$, which is
the Toeplitz analogue of
(\ref{eqn:kernel-szego-equiv-extended}); explicitly, it is given by
\begin{equation}
 \label{eqn:kernel-toeplitz-equiv-extended}
\mathcal{T}^\mu_k[f]\left(\lambda,\lambda'\right)=
\sum_j\,\widetilde{T^\mu_k[f]\left(s_j^{(k)}\right)}(\lambda)
\cdot \overline{\widetilde{s}_j^{(k)}(\lambda')}
\,\,\,\,\,\,\,\,\,\left(\lambda,\lambda'\in A^\vee_0\right),
\end{equation}
and satisfies the equivariance law (\ref{eqn:kernel-szego-equiv-extended-bullet}).
Corresponding to (\ref{eqn:key-relation-Pi-P}) we now have
\begin{eqnarray}
\label{eqn:toeplitz-key-relation-Pi-P}
 T^\mu_k[f]\big(\sigma_u(n),\sigma_u(n')\big)=e^{-\frac k2\,\big( \varXi(n)+\varXi(n')\big)}\,
\mathcal{T}^\mu_k[f]\left(\sigma (n),
\sigma (n')\right).
\end{eqnarray}

Let us define $K^\mu_k[f]:X\rightarrow \mathbb{R}$, the Toeplitz analogue of
(\ref{eqn:equivariant-distortion-fnct}), by setting
\begin{equation}
 \label{eqn:toeplitz-distortion-function}
K^\mu_k[f](x)=:T^\mu_k[f](x,x)=\sum_jT^\mu_k[f]\left(s_j^{(k)}\right)(x)\,
\overline{s_j^{(k)}(x)}\,\,\,\,\,\,\,\,\,
(x\in X)
\end{equation}
(since $f$ is real, $T^\mu_k[f]:H^\mu_k(X)\rightarrow H^\mu_k(X)$ is self-adjoint, and
so $T^\mu_k[f](x,x)\in \mathbb{R}$).

Then $K^\mu_k[f]$ descends to a $\nu$-invariant $\mathcal{C}^\varpi$ function 
on $N$, by an obvious analogue of Lemma
\ref{lem:analytic-distortion-function}, and so we can consider its unique
sesquiholomorphic extension
$\widetilde{K^\mu_k[f]}$ to a neighborhood of the diagonal in $N\times N$.
In place of (\ref{eqn:relation-sesquiholomorphic-extensions}) we now have that
\begin{equation}
 \label{eqn:toeplitz-relation-sesquiholomorphic-extensions}
\mathcal{T}^\mu_k[f]\left(\sigma (n),\sigma (n')\right)=
e^{k\,\widetilde{\varXi}(n,n')}\,\widetilde{K^\mu_k[f]}(n,n').
\end{equation}

Finally, a Toeplitz operator $T[f]=\Pi\circ M_f\circ \Pi$ is a zeroth order 
FIO associated to the same almost complex Lagrangian relation as $\Pi$, 
and therefore also has a microlocal structure of the form (\ref{eqn:szego-microlocal}),
with an amplitude having an asymptotic expansion as in (\ref{eqn:szego-microlocal-amplitude}).
Repeating the arguments following (\ref{eqn:equivariant-szego-kernel-asymptotics}), 
therefore, leads to the Toeplitz generalization
of the asymptotic expansion (\ref{eqn:asymptotic-expansion-Pi-k-n-n'-1})
and (\ref{eqn:asymptotic-expansion-K-k-mu}):
\begin{eqnarray}
 \label{eqn:toeplitz-asymptotic-expansion-Pi-k-n-n'-1}
T^\mu_k[f]\big(\sigma_u(n),\sigma_u(n')\big)&\sim&
\left(\frac k\pi\right)^d\,e^{k\,\left[
\widetilde{\varXi}(n,n')-\frac 12\,\big( \varXi(n)+\varXi(n')\big)\right]
}\nonumber\\
&&\cdot\sum_{j\ge 0}k^{-j}\,\widetilde{S^\mu_j[f]}(n,n'),
\end{eqnarray}
\begin{equation}
 \label{eqn:toeplitz-asymptotic-expansion-K-k-mu}
\widetilde{K^\mu_k}[f](n,n')\sim \left(\frac k\pi\right)^d\,\sum_{j\ge 0}k^{-j}\,\widetilde{S^\mu_j[f]}(n,n')
\end{equation}

Let us prove Theorem \ref{thm:heisenberg-relation}.

\begin{proof}
Let us adopt the notation and setting of the proof of Theorem
\ref{thm:lower-order-terms}.
 Given (\ref{eqn:toeplitz-asymptotic-expansion-Pi-k-n-n'-1}), arguing as in the derivation
of (\ref{eqn:composition-idempotent-parametrized-2}) we now obtain
\begin{eqnarray}
 \label{eqn:composition-toeplitz-parametrized-2}
\lefteqn{E^\mu_k[f,g](x_0,x_0)=\Big(T^\mu_k[f]\circ T^\mu_k[g]\Big)(x_0,x_0)}\\
&\sim&\left(\frac k\pi\right)^{2d}\,\sum_{j\ge 0}k^{-j}\int_B
e^{-k\,\mathcal{D}_N(n_0,n)}\,Z_j[f,g](n_0,n)
\,\det \big([\eta'_{k\overline{l}}]\big)\,dx\,dy,\nonumber
\end{eqnarray}
where now
\begin{eqnarray}
 \label{eqn:definition-Zj-toeplitz}
\lefteqn{Z_j[f,g](n,n')}\\
&=:&\Phi(n')^{d+1}\,\sum_{a+b=j}\widetilde{S_a^\mu[f]}(n,n')\,
\widetilde{S_b^\mu[g]}(n',n)\,\,\,\,\,\,\,\,\,\,\,
\big((n,n')\in V\times V\big).\nonumber
\end{eqnarray}

Corresponding to (\ref{eqn:composition-idempotent-parametrized-3}), we now have
\begin{eqnarray}
 \label{eqn:composition-toeplitz-parametrized-3}
\lefteqn{E^\mu_k[f,g](x_0,x_0)
\sim\left(\frac k\pi\right)^{d}\,\sum_{j,l\ge 0}k^{-j-l}\,
\left.R_l^N\Big(Z_j[f,g](n_0,\cdot)\Big)\right|_{n=n_0}}\\
&=&\left(\frac k\pi\right)^{d}\,\sum_{j\ge 0}k^{-j}\,
\sum_{a+b=j}\left.R_a^N\Big(Z_j[f,g](n_0,\cdot)\Big)\right|_{n=n_0}\nonumber\\
&=&\left(\frac k\pi\right)^{d}\,\Big\{Z_0[f,g](n_0,n_0)+k^{-1}\,A_1[f,g](n_0)
+O\left(k^{-2}\right)
\Big\},\nonumber
\end{eqnarray}
where
\begin{equation}
\label{eqn:defn-A-1}
A_1[f,g](n_0)=: \left.R_1^N\Big(Z_0[f,g](n_0,\cdot)\Big)\right|_{n=n_0}+Z_1[f,g](n_0,n_0)
\end{equation}

Now, by (\ref{eqn:definition-Zj-toeplitz}), we have
\begin{eqnarray}
 \label{eqn:Z-0-toeplitz}
Z_0[f,g](n_0,n_0)&=&\Phi(n_0)^{d+1}\,S_a^\mu[f](n_0)\,
S_b^\mu[g](n_0)\nonumber\\
&=&\Phi(n_0)^{-(d+1)}\,f(n_0)\,g(n_0)=Z_0[g,f](n_0,n_0);
\end{eqnarray}
therefore, 
\begin{eqnarray}
 \label{eqn:E-1-commutator}
\lefteqn{E^\mu_k[f,g](x_0,x_0)-E^\mu_k[g,f](x_0,x_0)}\nonumber\\
&=&\left(\frac k\pi\right)^d\,
\left[k^{-1}\Big(A_1[f,g](n_0)-A_1[g,f](n_0)\Big)+O\left(k^{-2}\right)\right].
\end{eqnarray}

Furthermore, by (\ref{eqn:definition-Zj-toeplitz}) we have
\begin{eqnarray}
 \label{eqn:Z-1-toeplitz}
Z_1[f,g](n_0,n_0)&=&\Phi(n_0)^{d+1}\,\Big[S_0^\mu[f](n_0)\,
S_1^\mu[g](n_0)+S_1^\mu[f](n_0)\,
S_0^\mu[g](n_0)\Big]\nonumber\\
&=&Z_1[g,f](n_0,n_0).
\end{eqnarray}
We see from (\ref{eqn:defn-A-1}) and (\ref{eqn:Z-1-toeplitz}) that
\begin{eqnarray}
 \label{eqn:A_1-commutator}
\lefteqn{A_1[f,g](n_0)-A_1[g,f](n_0)}\\
&=&\left.R_1^N\Big(Z_0[f,g](n_0,\cdot)\Big)\right|_{n=n_0}-
\left.R_1^N\Big(Z_0[g,f](n_0,\cdot)\Big)\right|_{n=n_0}\nonumber\\
&=&\left.\Delta_N\Big(Z_0[f,g](n_0,\cdot)\Big)\right|_{n=n_0}-
\left.\Delta_N\Big(Z_0[g,f](n_0,\cdot)\Big)\right|_{n=n_0};\nonumber
\end{eqnarray}
in the latter equality we have used that $R_1^N=\Delta_N-\varrho_N/2$ and
(\ref{eqn:Z-0-toeplitz}).

To compute the latter commutator, let us remark that
\begin{eqnarray}
 \label{eqn:Z-0-toeplitz-general}
Z_0[f,g](n_0,n)&=&\Phi(n)^{d+1}\,\widetilde{S_0^\mu[f]}(n_0,n)\,
\widetilde{S_0^\mu[g]}(n,n_0)\nonumber\\
&=&Z_0(n_0,n)\,\widetilde{f}(n_0,n)\,\widetilde{g}(n,n_0),
\end{eqnarray}
where $Z_0(n_0,n)$ is as in (\ref{eqn:definition-Z0}).

It follows from (\ref{eqn:Z-0-toeplitz-general}) and Lemma \ref{lem:technical-1}
that
\begin{eqnarray}
 \label{eqn:commutator-laplacian-term-A-1}
\left.\Delta_N\Big(Z_0[f,g](n_0,\cdot)\Big)\right|_{n=n_0}
&=&
\left.\Delta_N\Big(Z_0(n_0,\cdot)\Big)\right|_{n=n_0}\cdot f(n_0)\,g(n_0)\\
&&+
Z_0(n_0,n_0)\,\left.\Delta_N
\Big(\widetilde{f}(n_0,\cdot)\,\widetilde{g}(\cdot,n_0)\Big)\right|_{n=n_0}.\nonumber
\end{eqnarray}

Let $\big({h'}^{\overline{r}s}\big)$ be the contravariant metric tensor of
$(N',I,\eta')$, where $\eta'=2\,\eta$ (thus ${h'}^{\overline{r}s}=h^{\overline{r}s}/2$).
Since the former summand on the right hand side of (\ref{eqn:commutator-laplacian-term-A-1})
is symmetric in $f$ and $g$, we have
\begin{eqnarray}
 \label{eqn:A_1-commutator-0}
\lefteqn{A_1[f,g](n_0)-A_1[g,f](n_0)}\nonumber\\
&=&\Phi(n_0)^{-(d+1)}\,\left[\left.\Delta_N
\Big(\widetilde{f}(n_0,\cdot)\,\widetilde{g}(\cdot,n_0)\Big)\right|_{n=n_0}-
\left.\Delta_N
\Big(\widetilde{g}(n_0,\cdot)\,\widetilde{f}(\cdot,n_0)\Big)\right|_{n=n_0}\right]\nonumber\\
&=&\Phi(n_0)^{-(d+1)}\,{h'}^{\overline{r}s}\,
\Big(\partial_{\overline{r}}f (n_0)\,\partial_s g(n_0)-\partial_{\overline{r}}g (n_0)\,\partial_s f(n_0)\Big)
\nonumber\\
&=&-i\,\Phi(n_0)^{-(d+1)}\,\{f,g\}_N,
\end{eqnarray}
where in the latter step we have used (\ref{eqn:poisson-bracket-kahler}).
The last equality in the statement now follows from (\ref{eqn:A_1-commutator-0})
and Corollary \ref{cor:poisson-brackets-MN}.

\end{proof}

\end{document}